\documentclass[a4paper,11pt,leqno]{article}

\usepackage[latin1]{inputenc}
\usepackage[T1]{fontenc} 
\usepackage[english]{babel}
\usepackage{verbatim}
\usepackage{calligra}
\usepackage{enumerate}
\usepackage{dsfont}
\usepackage{amsfonts}
\usepackage{amsmath}
\usepackage{amsthm}
\usepackage{amssymb}
\usepackage{mathtools}
\usepackage{mathrsfs}
\usepackage{color}
\usepackage{graphicx}
\usepackage{bm}
\usepackage{srcltx}

\usepackage{tikz}
\usepackage[retainorgcmds]{IEEEtrantools}

\usepackage{graphicx}		
\usepackage[hypcap=false]{caption}

\theoremstyle{definition}
\newtheorem{defi}{Definition}[section]

\newtheorem{rmk}[defi]{Remark}

\theoremstyle{plain}
\newtheorem{thm}[defi]{Theorem}
\newtheorem{prop}[defi]{Proposition}

\newtheorem{lemma}[defi]{Lemma}

\newcommand{\tbf}{\textbf}

\newcommand{\tsl}{\textsl}

\newcommand{\mbb}{\mathbb}
\newcommand{\mbf}{\mathbf}
\newcommand{\mc}{\mathcal}

\newcommand{\mds}{\mathds}

\newcommand{\veps}{\varepsilon}
\newcommand{\eps}{\veps}
\newcommand{\what}{\widehat}
\newcommand{\wtilde}{\widetilde}
\newcommand{\vphi}{\varphi}
\newcommand{\oline}{\overline}

\newcommand{\s}{\sigma}
\renewcommand{\t}{\tau}

\newcommand{\de}{\delta}
\renewcommand{\o}{\omega}

\newcommand{\lan}{\langle}
\newcommand{\ran}{\rangle}

\newcommand{\R}{\mathbb{R}}

\newcommand{\N}{\mathbb{N}}
\newcommand{\Z}{\mathbb{Z}}
\newcommand{\T}{\mathbb{T}}

\newcommand{\D}{\mathbb{D}}

\renewcommand{\div}{{\rm div}\,}

\newcommand{\dx}{ \, {\rm d} x}
\newcommand{\dt}{ \, {\rm d} t}

\newcommand{\al}{\alpha}
\newcommand{\bt}{\beta}

\allowdisplaybreaks

\def\d{\partial}
\def\div{{\rm div}\,}
\newcommand{\dd}{\,{\rm d}}

\newcommand{\norm}[1]{\left\lVert#1\right\rVert}
\newcommand{\seminorm}[1]{\left\lvert#1\right\rvert}
\newcommand{\comm}[2]{\left[ #1,#2 \right]}

\newcommand{\Sg}{\Sigma}

\begin{document}

\newcommand{\fra}[1]{\textcolor{blue}{* FRA: #1 *}}

\title{\textsc{\Large{\textbf{Well-posedness of the Kolmogorov two-equation model of turbulence in optimal Sobolev spaces}}}}

\author{\normalsize \textsl{Oph\'elie Cuvillier}$\,^{1,}\footnote{Present affiliation: Lycée Joffre -- 150, All\'ee de la Citadelle, F-34060 Montpellier cedex 2, FRANCE.}\;,\qquad$
\textsl{Francesco Fanelli}$\,^2\qquad$ and $\qquad$
\textsl{Elena Salguero}$\,^{3}$ \vspace{.5cm} \\
\footnotesize{$\,^{1,2}\;$ \textsc{Universit\'e de Lyon, Universit\'e Claude Bernard Lyon 1}} \\
{\footnotesize \it Institut Camille Jordan -- UMR 5208}\\
{\footnotesize 43 blvd. du 11 novembre 1918, F-69622 Villeurbanne cedex, FRANCE} \vspace{.3cm} \\
%
%
\footnotesize{$\,^{3}\;$ \textsc{ICMAT \& Universidad de Sevilla}} \\
{\footnotesize Avda. Reina Mercedes s/n, CITIUS II, 41012 Sevilla, SPAIN} \vspace{.3cm} \\
\footnotesize{\hspace{-0.7cm}
Email addresses: $\,^{1}\;$\ttfamily{ophelie.cuvillier@ac-montpellier.fr}},
\footnotesize{$\,$ $^{2}\;$\ttfamily{fanelli@math.univ-lyon1.fr}},
\footnotesize{$\,$ $^{3}\;$\ttfamily{esalguero@us.es}}
\vspace{.2cm}
}

\date\today

\maketitle

\subsubsection*{Abstract}
{\footnotesize

In this paper, we study the well-posedness of the Kolmogorov two-equation model of turbulence in a periodic domain $\T^d$, for space dimensions $d=2,3$.
We admit the average turbulent kinetic energy $k$ to vanish in part of the domain, \tsl{i.e.} we consider the case $k \ge 0$;
in this situation, the parabolic structure of the equations becomes degenerate.

For this system, we prove a local well-posedness result in Sobolev spaces $H^s$, for any $s>1+d/2$. We expect this regularity to be optimal,
due to the degeneracy of the system when $k \approx 0$. We also prove a continuation criterion and provide a lower bound for the lifespan of the solutions.
The proof of the results is based on Littlewood-Paley analysis and paradifferential calculus on the torus, together with a precise commutator decomposition
of the non-linear terms involved in the computations.

}

\paragraph*{\small 2020 Mathematics Subject Classification:}{\footnotesize 35Q35 
(primary);
76F60, 
35B65 
(secondary).}

\paragraph*{\small Keywords: }{\footnotesize Kolmogorov two-equation model of turbulence; local well-posedness; degenerate parabolic effect;
commutator structure.}


\section{Introduction and main results} \label{s:intro}

In this paper, we study a system of PDEs which was proposed by Kolmogorov \cite{Kolm_42} (see the Appendix of \cite{Spald} for an English translation)
to describe fluid flows in a fully developed isotropic turbulent regime.

\subsection{The system of equations}
As in other (yet more recent) one-equation or two-equation models (see \tsl{e.g.} \cite{launder1972lectures} and \cite{Mo-Pi} for details and more references about
the latters, \cite{CR-Lew} about the formers), the Kolmogorov model postulates that one can identify related, but somehow independent variables to describe
the large scale behaviour (\tsl{i.e.} the mean motion) of the fluid and the small scale fluctuations (\tsl{i.e.} the turbulent character).
Here, the average has to be intended always in a statistical sense, namely as an \emph{ensemble average}, although Kolmogorov
seemed to refer to time average in his original paper.

Thus, let $t\in\R_+$ denote the time variable and $x\in\Omega\subset\R^d$ be the space variable, with $d\geq2$ and $\Omega$ being a smooth domain.
Define $u=u(t,x)\in\R^d$ to be the mean velocity field of the fluid, $\o=\o(t,x)\geq0$ the mean frequency of turbulent fluctuations
and $k=k(t,x)\geq0$ the mean turbulent kinetic energy, that is, the kinetic energy associated to the variations of the velocity field from its mean value $u$.
Then, the Kolmogorov model \cite{Kolm_42} reads
\begin{equation} \label{eq:kolm_d}
\left\{\begin{array}{l}
       \d_tu\,+\,(u\cdot\nabla) u\,+\,\nabla\pi\,-\,\nu\,\div\left(\dfrac{k}{\o}\,\D u\right)\,=\,0 \\[1ex]
       \d_t\o\,+\,u\cdot\nabla\o\,-\,\alpha_1\,\div\left(\dfrac{k}{\o}\,\nabla\o\right)\,=\,-\,\alpha_2\,\o^2 \\[1ex]
       \d_tk\,+\,u\cdot\nabla k\,-\,\alpha_3\,\div\left(\dfrac{k}{\o}\,\nabla k\right)\,=\,-\,k\,\o\,+\,\alpha_4\,\dfrac{k}{\o}\,\big|\D u\big|^2 \\[2ex]
       \div u\,=\,0\,.
       \end{array}
\right.
\end{equation}
The flow is assumed to be homogeneous, thus incompressible, whence the last equation appearing in the system. The function $\pi=\pi(t,x)\in\R$
represents the pressure field of the fluid; its gradient $\nabla\pi$ can be interpreted as a lagrangian multiplier associated to the
divergence-free constraint (\tsl{i.e.} to the incompressibility condition). The symbol $\D$ appearing in the first and third equations
stands for the symmetric part of the gradient of $u$:
\[
 \D u\,:=\,\frac{1}{2}\,\big(D u\,+\,\nabla u\big)\,,
\]
where we have denoted by $Du$ the Jacobian matrix of $u$ and by $\nabla u$ its transpose matrix. 
Finally, the quantities $\nu,\al_1,\ldots \al_4$ are strictly positive numbers, which represent physical adimensional parameters; in \cite{Kolm_42}, Kolmogorov
even gave explicit values for some of them.
Notice that, in system \eqref{eq:kolm_d}, we have assumed that no external forces are acting on the fluid.

We do not enter into the discussion of the physical explanation or motivation of system \eqref{eq:kolm_d}. 
Let us simply point out that equations \eqref{eq:kolm_d} seem to retain retain one of the main aspects of turbulence theory, namely
the transfer of energy from large scales to smallest scales through viscous dissipation. This is exactly the meaning
of the presence of the $\alpha_4$-term in the third equation.
We refer to books \cite{Frisch}, \cite{Les} and \cite{Dav} about this matter and many other theoretical aspects linked to turbulence in fluids.

\subsection{Overview of the related literature}

Interestingly, in Kolmogorov's model the fluid is assumed to be, to the best of our understanding, inviscid.
As a matter of fact, the only viscosity which appears in the equations
is the so-called \emph{eddy viscosity} of Boussinesq (see Chapter 4 of \cite{Dav} for more details about this), which takes the form
\[
 \nu_{\rm eddy}\,=\,\frac{k}{\o}\,.
\]
As already said, for a physical insight on turbulence theory we refer to the previously mentioned books. Here, we rather comment on the mathematical
properties of the Kolmogorov system.

Despite the absence of a ``true'' viscosity, the eddy viscosity $\nu_{\rm eddy}$ endows system \eqref{eq:kolm_d} of a nice parabolic structure as soon as
$k>0$. On the other hand, an underlying maximum principle for those equations allows to establish that, if the mean turbulent kinetic energy
is strictly positive initially, namely if
\begin{equation} \label{eq:k_0>0}
 k_0\,\geq\,k_*\,>\,0\qquad\qquad \mbox{ on }\qquad \Omega\,,
\end{equation}
then for any later time $t>0$, one has $k_{\min}(t)\,:=\,\min_{x\in\Omega}k(t,x)\,>\,0$ (see more details in Subsection \ref{ss:energy} below).
Hence, condition \eqref{eq:k_0>0} ensures the preservation of the above mentioned parabolic structure.

Thus, first mathematical studies on well-posedness of the Kolmogorov model \eqref{eq:kolm_d} focused on the situation in which condition \eqref{eq:k_0>0} holds true.
In particular, in \cite{Mie-Nau} (see also \cite{Mie-Nau_CRAS} for an announcement of the result), Mielke and Naumann proved the existence of global in time finite
energy weak solutions to \eqref{eq:kolm_d} in the periodic three dimensional box $\T^3$, under condition \eqref{eq:k_0>0}.
The same condition was used by Kosewski and Kubica to set down a strong solutions theory, see \cite{Kos-Kub} and \cite{Kos-Kub_glob} for, respectively,
a local well-posedness result and a global well-posedness result for small initial data (see also \cite{Kos} by Kosewski for an extension
to the case of fractional regularities).

At this point, we observe that, from the physical viewpoint, condition \eqref{eq:k_0>0} looks somehow a bit restrictive. For instance, we quote from the
introduction of \cite{Mie-Nau}:
\begin{quotation}
``It would be desirable to develop an existence theory without this condition, because this would allow us to study how the support of $k$, which is may be called the turbulent region, invades the non-turbulent region where $k \equiv 0$''.
\end{quotation}
In other words, taking into account the possible vanishing of the mean turbulent kinetic energy $k$ may help in the description and understanding of the transition
from turbulent to non-turbulent regimes, and \tsl{viceversa}.

Nonetheless, very few results seem to deal with a situation in which assumption \eqref{eq:k_0>0} is not considered.
For instance, in \cite{Bul-Mal} Bul\'i\v{c}ek and M\'alek studied system \eqref{eq:kolm_d} in a smooth bounded domain $\Omega\subset\R^3$; under the conditions
\[
 k_0\,>\,0\qquad \mbox{ in }\quad \Omega\,,\qquad\qquad\qquad \log k_0\,\in\,L^1(\Omega)\,,
\]
they established the existence of global in time finite energy weak solutions, similarly in spirit to the result of \cite{Mie-Nau}. It is worth to point out
that, however, many differences in the analysis arise between the two works \cite{Mie-Nau} and \cite{Bul-Mal}: for instance, in the latter reference
non-trivial boundary conditions are taken into account, thus allowing for a description of boundary-induced turbulent phenomena. We avoid to comment more about
the specific contents of the two papers here, as this discussion would go beyond the scopes of our presentation.

More recently, in work \cite{F-GB} the authors considered a one-dimensional reduction of the Kolmogorov system \eqref{eq:kolm_d} and investigated its well-posedness
in the torus $\T^1$ in the generic situation $k_0\geq0$. As a matter of fact, some mild degeneracy assumptions for $k_0$ close to the ``vacuum region'' $\big\{k_0=0\big\}$
have to be assumed, in the sense that $\sqrt{k_0}$ must be regular enough. Under this condition, the authors established, on the one hand,
the existence and uniqueness of local in time regular solutions to the $1$-D model and, on the other hand, the existence of smooth initial profiles which give rise
to solutions which blow up in finite time. These results were later extended in \cite{F-GB_ZAMP} for a class of toy-models
introduced in \cite{Mie} (see also the introduction of \cite{F-GB} for the discussion of a specific toy-model).

\subsection{Statement of the main results}

The results of \cite{F-GB} constitute the starting point of the present work. We observe that, in that paper, the well-posedness result was stated only
for integer regularity indices $H^m$, with $m\in\N$ and $m\geq2$. In addition, the blow-up mechanisms highlighted in \cite{F-GB} and \cite{F-GB_ZAMP}
seem to be quite specific to the one-dimensional situation.

These remarks are the main motivation for our study. 
While it is not clear, at present, whether or not the blow-up results of \cite{F-GB} and \cite{F-GB_ZAMP} may be extended to higher dimensions,
in the present paper we generalise the local well-posedness result of \cite{F-GB} in two aspects: first of all, we extend it to the physically relevant situation
of two and three-dimensional flows;
in addition, we prove well-posedness in optimal Sobolev spaces $H^s(\T^d)$, with $s>1+d/2$ and $d=2,3$ (in fact, the result is stated for a generic dimension
$d\geq2$). Here, ``optimal'' refers to both minimal regularity
and integrability. 

Let us comment a bit on the previous sentence. First of all, we observe that, because of the appearing of transport terms in equations \eqref{eq:kolm_d},
we need to solve the system in a functional framework able to guarantee a $L^1_T(L^\infty)$ control for  the gradient of the velocity field.
Unfortunately, the degeneracy of the parabolic character of the equations when $k\approx0$ prevents us from using any kind of smoothing property in the
dynamics. From this point of view, then, it is natural to look for well-posedness results in spaces $H^s$ such that $s\,>\,s_0\,:=\,1+d/2$ or, more in general,
in Besov spaces $B^s_{p,r}$, with $s>1+d/p$ and $r\in[1,+\infty]$ up to the endpoint case $s=1+d/p$ and $r=1$.
On the other hand, even in the integer case $s=m\in\N$, $m>1+d/p$, using the degenerate parabolic smoothing seems to be necessary in order to close the estimates
for the higher order norms of the solution. However, in order to do that, one needs to use integration by parts and the symmetric structure of the viscosity term,
a fact which forces us to take $p=2$ in the previous conditions.
We refer to \cite{F-GB} for more explanations about this.

With this considerations in mind, we can state the first main result of the paper, which contains local existence and uniqueness of solutions in $H^s$, for
any $H^s$ initial datum. Inspired by \cite{F-GB}, conditions are formulated on $\sqrt{k}$ instead of $k$ itself.
Throughout this work, we set equations \eqref{eq:kolm_d} in the $d$-dimensional torus
\[
 \Omega\,=\,\T^d\,,\qquad\qquad \mbox{ with }\qquad d\geq2\,.
\]
The precise statement is the following one.

\begin{thm} \label{t:d_wp}
Let $s > 1 + d/2$. Take any triplet $(u_0, \o_0, k_0)$ of functions satisfying the following assumptions:
\begin{enumerate}[(i)]
	\item $u_0, \o_0 \,\in\, H^s(\Omega)$, with the divergence-free constraint $\div u_0 = 0$;
	\item there exist two constants $0<\o_*\leq\o^*$ such that $\o_\ast \leq \o_0 \leq \o^\ast$;
	\item $k_0 \ge 0$ is such that $\bt_0:= \sqrt{k_0}\,\in\, H^s(\Omega)$.
\end{enumerate}

Then, there exists a time $T>0$ such that system \eqref{eq:kolm_d}, equipped with the initial datum $(u_0,\o_0,k_0)$, admits a unique solution
$(u,\nabla\pi,\o,k)$ on $[0,T]  \times \Omega$ enjoying the following properties:
\begin{enumerate}[1)]
\item the functions $u, \o$ and $\sqrt{k}$ belong to the space $ L^\infty\big([0,T];H^s(\Omega)\big)\cap\,\bigcap_{\s<s} C\big([0,T]; H^\s(\Omega)\big)$;
\item the non-negativity of $\o$ and $k$ is propagated in time: for any $(t,x) \in [0,T] \times \Omega, \; \o(t,x) >0$ and $k(t,x) \ge 0$;
\item the gradient of the pressure $\nabla \pi$ belongs to $L^\infty\big([0,T]; H^{s-1}(\Omega)\big)\cap\,\bigcap_{\s<s} C\big([0,T]; H^{\s-1}(\Omega)\big)$;
\item the functions $\sqrt{\dfrac{k}{\o}} \D u $, $\sqrt{\dfrac{k}{\o}} \nabla \o $ and $\sqrt{\dfrac{k}{\o}} \nabla\sqrt{k} $ all belong to the
space $L^2\big([0,T];H^s(\Omega)\big)$.
\end{enumerate}
In addition, this solution $\big(u,\nabla\pi,\o,k\big)$ is unique within the class
	\begin{align*}
\mathbb{X}_T(\Omega)\,:=\,\Big\{ (u,\nabla\pi,\o,k) \; \Big| & \quad u\,,\, \o\,,\, \sqrt{k} \,\in\, C\big([0,T];L^2(\Omega)\big)\,, \quad
\nabla\pi\,\in\,L^\infty\big([0,T];H^{-1}(\Omega)\big)\,, \\
&\qquad\quad \o\,,\, \o^{-1}\,,\, k \,\in\, L^\infty\big([0,T]\times \Omega\big)\,, \quad \o >0\,, \quad k \ge 0\,, \\
&\qquad\qquad\qquad \div u\,=\,0\,, \qquad \nabla u\,,\, \nabla \o\,,\, \nabla\sqrt{k}\, \in\, L^\infty\big([0,T]\times \Omega\big)
\Big\}\,.
	\end{align*}
\end{thm}

The previous statement generalises the corresponding well-posedness result of \cite{F-GB} to the case of higher dimension $d\geq2$ and from the angle of
minimal regularity assumptions on the initial data. Notice that the passage from integer regularity indices to
fractional ones involves some technical difficulties, that we want now to discuss.

The first major difficulty we encounter in our analysis is related to the use of the degenerate parabolic regularisation effect on the solutions. In the case of
integer regularities $s=m\in\N$,  it is natural to see how to use this effect, as all the computations are explicit and factors $\sqrt{k}/\o$ can be easily
moved from one term to another, in order to make the right coefficient appear in front of the term with the highest number of derivatives.
In the case of fractional regularities, instead, differentiation of the equations is replaced by commutators with frequency-localisation operators (indeed, we will
broadly use Littlewood-Paley characterisation of Sobolev spaces $H^s$ in the torus).
Finding the right commutator structure of the equations, which enables us to use the degenerate parabolic smoothing, thus becomes rather involved and not
straightforward at all. We refer to Paragraph \ref{sss:commutators} below for more details, see in particular the splitting
of the commutator terms $\mds C^2[f,f]$, for $f\in\big\{u,\o,\sqrt{k}\big\}$, into the sum of four terms. In addition, the use of frequency-localisation
operators on the viscosity terms entails a control on a rather strange quantity $\mc S$ (which is a sum of norms of dyadic blocks) related to $f$,
which cannot be reconducted to the $H^{s+1}$ norm of $f$ because of the degeneracy of the viscosity/diffusion coefficient $k/\o$ when $k\approx 0$.
At this point, a key observation (contained in Proposition \ref{p:key} below) establishes the equivalence of this quantity $\mc S$ with the $H^s$ norm of
$\sqrt{k}/\o\,\nabla f$, up to lower order terms. This combines well with the above mentioned commutator structure, which indeed allows us each time to
put in evidence a factor $\sqrt{k}/\o$ in front of the term with the highest order derivatives.

One last point which should be mentioned in this context is the control of the Sobolev norm $H^s$ of the negative power $\o^{-1/2}$.
The problem is that the lower bound for $\o$ degenerates with time, namely $\min_\Omega\o(t)\,\longrightarrow\,0$ when $t\to+\infty$. This further degeneracy
prevents us from using classical paralinearisation theorems, as a very precise control of the $\norm{\o^{-1/2}}_{H^s}$ in terms of $\min_\Omega\o$ is needed
in the analysis. Of course, such a problem does not appear when $s=m$ is an integer, because in that case one disposes of explicit computations.
In the end, we will establish the required precise bound in Lemma \ref{l:comp} below.

\medbreak
After the previous comments on the statement of Theorem \ref{t:d_wp} and its proof, let us move forward. We now present the second main result of the paper,
which complements Theorem \ref{t:d_wp} with some information about the lifespan of solutions and with a continuation criterion.

The precise statement is the following one.

\begin{thm}\label{t:cont}
Let $s>1+d/2$. Take an initial datum $\big(u_0,\o_0,k_0\big)$ which verifies the assumptions of Theorem \ref{t:d_wp}. Let $\big(u,\nabla\pi,\o,k\big)$
be the corresponding unique solution satisfying the conditions stated in that theorem.
Denote by $T>0$ its lifespan.

Then, if we define the energy $\mc E_0$ of the initial datum as
\begin{equation*}
	\mathcal{E}_0\, :=\, \norm{u_0}_{H^s}^2\, +\, \norm{\o_0}_{H^s}^2\, +\, \norm{\sqrt{k_0}}_{H^s}^2\,.
\end{equation*}
there exists a constant $C = C(d,s,\nu, \alpha_1, \dots, \alpha_4, \o_*, \o^\ast) >0$, only depending on the quantities inside the brackets, such that 
\begin{equation} \label{est:lower-T}
T\, \ge\, \min \left\{  1\,,\, \frac{C}{\mc E_0\,\big(1+\mc E_0\big)^{2[s]+3}}\right\}\,,
\end{equation}
where the symbol $[s]$ stands for the integer part of $s$.

Furthermore, we have the following continuation/blow-up criterion. Let $T^*<+\infty$ such that the solution is well-defined in the time interval $[0,T^*[\,$.
Then, the $H^s$ norm of the solution becomes unbounded when $t\to T^*$ if and only if 
\begin{equation*}
	\int_0^{T^\ast} A(t)\,\dt \,=\,+\,\infty\,,
\end{equation*}
where we have defined
\begin{align*}
A(t)\, &:=\,\norm{\big(\nabla u,\nabla\o,\nabla\bt\big)}_{L^\infty}^{[s]+4} \\
&\qquad\qquad\qquad +\,\left(1+\norm{\nabla\bt}_{L^\infty}\right)\,\left(1+\norm{\nabla\o}_{L^\infty}^{[s]}\right)
\left(\sum_{G\in\{\D u,\nabla\o,\nabla\bt\}}\norm{ \nabla \left(\frac{\bt}{\sqrt \o}\,G\right) }_{L^\infty }\right)\,.
\end{align*}
\end{thm}

A few comments are in order. First of all, we notice that the definition of the function $A(t)$ in the continuation criterion looks more complicated than
the one appearing in the corresponding result from \cite{F-GB}. More precisely, the big sum on the second line is apparently missing in that reference.
The reason for this has to be ascribed to the involved commutator structure mentioned above and to the impossibility of moving the coefficients
$\sqrt{k}/\o$ freely from one term to another, which is instead possible when one simply differentiates the equations.

For somehow related reasons, also the lower bound \eqref{est:lower-T} for the lifespan of the solutions looks a bit different from the one established
in \cite{F-GB}. This bound is a direct consequence of inequality \eqref{est:ineq-for-T}, which would however allow us to establish more precise
estimates for $T$, at least in the two regimes $\mc E_0\ll1$ (in which case we expect $T\gg1$) and $\mc E_0\gg1$ (in which case we expect, conversely, $T\ll1$).

\subsection*{Organisation of the paper}

To conclude this introduction, we give an overview of the contents of the paper.

The next section is a toolbox. There, we review classical results from Fourier analysis and Littlewood-Paley theory on the torus, which will be needed in our analysis.
In particular, we also establish therein the above mentioned Proposition \ref{p:key} and Lemma \ref{l:comp}, which will play a fundamental role
in our study.

The following sections are devoted to the proof of the main results. In Section \ref{s:a-priori} we exhibit \tsl{a priori} estimates for smooth
solutions to system \eqref{eq:kolm_d}. At the end of the argument, we show the proof of Theorem \ref{t:cont}.
Section \ref{s:proof}, instead, is devoted to the proof of Theorem \ref{t:d_wp}. In particular, in a first time we show how
to deduce, from the \tsl{a priori} estimates of the previous section, existence of a solution at the claimed level of regularity.
Then, we derive uniqueness of solutions from a stability estimate in the energy space $L^2$.

\subsubsection*{Acknowledgements}

{\small
The work of the second author has been partially supported by the LABEX MILYON (ANR-10-LABX-0070) of Universit\'e de Lyon, within the program ``Investissement d'Avenir''
(ANR-11-IDEX-0007), and by the projects  SingFlows (ANR-18-CE40-0027) and CRISIS (ANR-20-CE40-0020-01), all operated by the French National Research Agency (ANR).

The work of the third author has been partially supported by the grant PRE2018-083984, funded by MCIN/AEI/ 10.13039/501100011033, by the ERC through the Starting Grant H2020-EU.1.1.-639227, by MICINN through the grants EUR2020-112271 and PID2020- 114703GB-I00 and by Junta de Andaluc\'ia through the grant P20-00566.

}

\section{Tools from Littlewood-Paley theory} \label{s:LP}
We present a summary of some fundamental elements of Littlewood-Paley theory and use them to derive some useful inequalities. We refer \tsl{e.g.} to Chapter 2 of \cite{B-C-D} for details on the construction in the $\R^d$ setting, to reference \cite{Danchin} for the adaptation to the case of a $d$-dimensional periodic box $\T^d_a$, where $a\in\R^d$
(this means that the domain is periodic in space with, for any $1\leq j\leq d$, period equal to $2\pi a_j$ with respect to the $j$-th component).

For simplicity of presentation, we focus here on the case in which all $a_j$ are equal to $1$.
We denote by
$\left|\T^d\right|\,=\,\mc L\big(\T^d\big)$ the Lebesgue measure of the box $\T^d$.

\medbreak
First of all, let us recall that, for a tempered distribution $u\in\mc S'(\T^d)$, we denote by $\mc Fu\,=\,\big(\what u_k\big)_{k\in\Z^d}$ its Fourier series,
so that we have
\[
u(x)\,=\,\sum_{k\in\Z^d}\what{u}_k\,e^{ik\cdot x}\,,\qquad\qquad\mbox{ with }\qquad \what u_k\,:=\,\frac{1}{\big|\T^d\big|}\int_{\T^d} u(x)\,e^{-ik\cdot x}\,dx\,.
\]

Next, we introduce the so called \textit{Littlewood-Paley decomposition} of tempered distributions. The Littlewood-Paley decomposition is based on a
non-homogeneous dyadic partition of unity with respect to the Fourier variable. In order to define it,
we fix a smooth scalar function $\vphi$ such that $0\leq \vphi\leq 1$, $\vphi$ is even and supported in the ring
$\left\{r\in\R\,\big|\ 5/6\leq |r|\leq 12/5 \right\}$, and such that
\[
\forall\;r\in\R\setminus\{0\}\,,\qquad\qquad \sum_{j\in\Z}\vphi\big(2^{-j}\,r\big)\,=\,1\,.
\]
Then, we define $|D|\,:=\,(-\Delta)^{1/2}$ as the Fourier multiplier\footnote{Throughout we agree  that  $f(D)$ stands for 
the pseudo-differential operator $u\mapsto\mc{F}^{-1}(f\,\mc{F}u)$, where $\mc F^{-1}$ is the inverse Fourier transform.} of symbol $|k|$, for $k\in\Z^d$.
The dyadic blocks $(\Delta_j)_{j\in\Z}$ are then defined by
$$
\forall\;j\in\Z\,,\qquad\qquad
\Delta_ju\,:=\,\varphi(2^{-j}|D|)u\,=\,\sum_{k\in\Z^d}\vphi(2^{-j}|k|)\,\what u_k\,e^{ik\cdot x}\,.
$$
Notice that, because we are working on a compactly supported set, one has that eventually, $\Delta_j\equiv0$ for $j<0$ negative enough (depending on the size of $\T^d_a$). In addition, one has the following Littlewood-Paley decomposition in $\mc S'(\T^d)$:
\begin{equation} \label{eq:LP-torus}
\forall\;u\in\mc{S}'(\T^d)\,,\qquad\qquad u\,=\,\what u_0\,+\, \sum_{j\in\Z}\Delta_ju\qquad\mbox{ in }\quad \mc{S}'(\T^d)\,.
\end{equation}
In the decomposition above, $\what u_0$ stands for the mean value of $u$ on $\T^d$, i.e.,
\[
\what u_0\,=\,\oline u\,=\,\,\frac{1}{\big|\T^d\big|}\int_{\T^d} u(x)\,dx\,.
\]


It is relevant to note that the Fourier multipliers $\Delta_j$ are linear operators which are bounded on $L^p$ for any $p\in[1,+\infty]$. In addition, their norms are \emph{independent} of both $j$ and $p$.

Littlewood-Paley decomposition can be used to characterise several classical functional spaces. For instance,
it is well known that Sobolev spaces $H^s(\T^d)$, for $s \in \R$, are characterised in terms of Littlewood-Paley decomposition (see Section 2.7 of \cite{B-C-D}) through the following equivalence of norms: 
\begin{equation} \label{eq:LP-Sob}
	\|u\|^2_{H^s}\,\sim\,\left|\what u_0\right|^2\,+\,\sum_{j\in\Z}2^{2sj}\,\|\Delta_ju\|_{L^2}^2\,.
\end{equation}
This characterisation involves the low order term $\seminorm{\what u_0}^2$. In fact, this term can be substituted by the square of the $L^2$ norm, by noticing that $|\what u_0|^2 \leq \norm{u}_{L^2}^2$: one thus has
\begin{equation}\label{eq:Hs}
	\norm{u}_{H^s}^2\, \sim\, \norm{u}_{L^2}^2\, + \,\sum_{j\in\Z}2^{2sj}\,\|\Delta_ju\|_{L^2}^2\,.
\end{equation}
For later use, let us observe that the high order term is equivalent to the homogeneous Sobolev norm of regularity $s$, namely
\begin{equation*} 
\|u\|^2_{\dot{H}^s}\,\sim \,\sum_{j\in\Z}2^{2sj}\,\|\Delta_ju\|_{L^2}^2\,.
\end{equation*}

Next, let us present a version of the classical \emph{Bernstein inequalities} adapted to our functional framework (see Chapter 2 of \cite{B-C-D} for a general statement of this result).
\begin{lemma} \label{l:bern}
 There exists a universal constant $C>0$, only depending on the size of the torus $\T^d$ and on the support of the function $\vphi$ defined above, such that
for any $j\in\Z$,
for any $m\in\N$, for any couple $(p,q)$ such that $1 \leq p \leq q \leq + \infty$,  and for any smooth enough $u\in \mc S'(\T^d)$, it holds
\begin{align*}
&\left\|\Delta_ju\right\|_{L^q}\, \leq\,
 C\,2^{jd\left(\frac{1}{p}-\frac{1}{q}\right)}\,\|\Delta_ju\|_{L^p} \\[1ex]
&\qquad\qquad\qquad\mbox{ and }\qquad\qquad
C^{-m-1}\,2^{-jm}\,\|\Delta_ju\|_{L^p}\,\leq\,
\left\|D^m \Delta_ju\right\|_{L^p}\,\leq\,C^{m+1} \, 2^{jm}\,\|\Delta_ju\|_{L^p}\,.
\end{align*}
\end{lemma}

Now, we apply the Littlewood-Paley decomposition and the Bernstein inequalities to deduce the following useful inequality,
similar in spirit to the Poincar\'e-Wirtinger type inequality:
\begin{equation} \label{est:interp_2}
 \forall\,f\in W^{1,\infty}(\T^d)\quad \mbox{ such that }\;\oline f\,=\,0\,,\qquad\qquad
\left\|f\right\|_{L^\infty}\,\lesssim\,\left\|f\right\|_{L^2}^{2/(d+2)}\;\left\| \nabla f\right\|_{L^\infty}^{d/(d+2)}.
\end{equation}
The proof relies on an optimization procedure for the dyadic partition of $f$ and the systematic use of Bernstein inequalities. In particular, for $N\in\N$ to be fixed later, we can estimate
\begin{align*}
\|f\|_{L^\infty}\,&\leq\,\sum_{j<0}\left\|\Delta_jf\right\|_{L^\infty}\,+\,\sum_{j=0}^N\left\|\Delta_jf\right\|_{L^\infty}\,+\,
\sum_{j\geq N+1}\left\|\Delta_jf\right\|_{L^\infty} \\
&\lesssim\,\sum_{j<0}2^{jd/2}\,\left\|\Delta_jf\right\|_{L^2}\,+\,\sum_{j=0}^N2^{jd/2}\,\left\|\Delta_jf\right\|_{L^2}\,+\,\sum_{j\geq N+1}2^{-j}\,2^{j}\,
\left\| \Delta_j f\right\|_{L^\infty} \\
&\lesssim\,\left(1\,+\,2^{Nd/2}\right)\,\|f\|_{L^2}\,+\,2^{-N}\,\| \nabla f\|_{L^\infty}\,.
\end{align*}
Now, we can choose $N$ such that
\[
2^{Nd}\,\|f\|_{L^2}^2\,\approx\,2^{-2N}\,\| \nabla f\|_{L^\infty}^2\qquad\qquad\Longrightarrow\qquad\qquad
2^{N}\,\approx\,\left(\frac{\| \nabla f\|_{L^\infty}}{\|f\|_{L^2}}\right)^{2/(d+2)}.
\]
Inequality \eqref{est:interp_2} follows immediately from the previous choice of $N$.


\medbreak
In the next section, we will present \tsl{a priori} estimates for smooth solutions to our system \eqref{eq:kolm_d}.
Those estimates will be essentially based on energy methods. However, owing to the non-linearities appearing in the equations,
the estimates of the higher order Sobolev norms of the solutions will involve some commutators.
In particular, the structure described in the following lemma will be present through all the estimates.
\begin{lemma}\label{l:comm}
Let $s >0$ and $d\geq 1$. Let $f$ be a scalar function and $u$ a $d$-dimensional vector field, both defined over $\T^d$.
There exists a constant $C= C(s,d)>0$, only depending on the quantities inside the brackets, such that 
	\begin{equation} \label{est:comm1}
\left(	\sum_{j \in \Z} 2^{2js} \norm{\comm{\Delta_j}{u} \cdot  \nabla f   }_{L^2}^2   \right)^{1/2}\, \leq\, C\,
\Big(\norm{\nabla u}_{L^\infty} \norm{f}_{H^s} + \norm{\nabla f}_{L^\infty} \norm{\nabla u}_{H^{s-1}} \Big)\,.
	\end{equation}
In addition, if $d\geq 2$ and $\div u = 0$, 
then one has
	\begin{equation} \label{est:comm2}
\left(	\sum_{j \in \Z} 	2^{2js} \norm{ \div\big( \comm{\Delta_j}{u} f  \big) }_{L^2}^2   \right)^{1/2}\, \leq\, C\,
\Big( \norm{\nabla u}_{L^\infty} \norm{f}_{H^s} + \norm{\nabla f}_{L^\infty} \norm{\nabla u}_{H^{s-1}} \Big)\,.
\end{equation}
\end{lemma}

\begin{proof}
Estimates \eqref{est:comm1} and \eqref{est:comm2} are particular cases of Lemma 2.100 in \cite{B-C-D}.
We adapt the proof in the previous reference (performed in the whole space case) to the geometry of the torus.

For that purpose, we show that the estimates do not depend on the mean values of $u$ and $f$. As a matter of fact, keeping \eqref{eq:LP-torus} in mind,
we can write 
	\begin{equation*}
u \,=\,\oline{u}\,+\,\wtilde{u} \qquad\qquad \mbox{ and }\qquad\qquad f\, =\,\oline{f}\,+\,\wtilde{f}\,,
	\end{equation*}
where $\wtilde{u}$ and $\wtilde{f}$ are functions with zero mean over $\T^d$. Then, as both $\oline{f}$ and $\oline u$ are real numbers,
we have $\nabla \oline f\equiv0$ and $\comm{\Delta_j}{\oline u}\equiv0$, which in turn implies the equality
\begin{equation*}
	\comm{\Delta_j}{u} \cdot \nabla f\,=\,\comm{\Delta_j}{ \wtilde{u}} \cdot \nabla \wtilde{f}\,.
\end{equation*}

Therefore, without loss of generality, we can assume  that the functions $u$ and $f$ have zero mean. Now, the proof of \eqref{est:comm1}
easily reduces to the one given in \cite{B-C-D}.
In addition, we notice that, if $\div u =0$, then we have the identity
$$
 \div\big( \comm{\Delta_j}{u} f \big)\,=\,\comm{\Delta_j}{u} \cdot  \nabla f\,.$$
In particular, when $\div u=0$, estimate \eqref{est:comm2} immediately follows from \eqref{est:comm1}.
\end{proof}

Remark that, in inequality \eqref{est:comm1}, the structure of the scalar product $u\cdot\nabla f$ is not really used.
In particular, the same inequality applies to any scalar functions $\alpha$ and $f$: for any $k\in\{1\ldots d\}$, one has
\begin{equation} \label{est:comm3}
\sum_{j \in \Z} 2^{2js} \norm{\comm{\Delta_j}{\alpha}  \d_k f   }_{L^2}^2 \, \leq\, C\,
\Big(\norm{\nabla\alpha}_{L^\infty} \norm{f}_{H^s} + \norm{\nabla f}_{L^\infty} \norm{\nabla \alpha}_{H^{s-1}} \Big)^2\,.
\end{equation}
In light of this observation, we can establish the next result, which will play a key role in our analysis.

\begin{prop} \label{p:key}
Let $d\geq1$ and $s>1+d/2$.
Take two scalar functions $\alpha$ and $f$ defined over $\T^d$, 
both belonging to $H^s(\T^d)$.
Let $P(\d)$ be a differential operator of order $1$ with constant coefficients.
Assume that 
\[
\mc S_s\big[\alpha,P(\d)f\big]\,:=\,
\sum_{j\in \Z}2^{2js}\,\int_{\T^3}\alpha^2\,\left|\Delta_jP(\d) f\right|^2\,\dx\,<\,+\,\infty\,.
\]

Then, the product $\alpha\, P(\d)f$ belongs to $H^{s}(\T^d)$. In addition, one has the following ``equivalence of norms modulo lower order terms'':
\begin{align*}
\left\|\alpha\,P(\d)f\right\|_{H^{s}}^2\,
&\lesssim\,\mc S_s\big[\alpha,P(\d)f\big]
\,+\,
\Big( \norm{\nabla f }_{L^\infty} \norm{\alpha}_{H^s}\, +\, \norm{\nabla \alpha}_{L^\infty} \norm{f}_{H^s}\Big)^2\,, \\
\mc S_s\big[\alpha,P(\d)f\big]\,
&\lesssim\,\left\|\alpha\,P(\d)f\right\|_{\dot H^{s}}^2
\,+\,
\Big( \norm{\nabla f }_{L^\infty} \norm{\alpha}_{H^s}\, +\, \norm{\nabla \alpha}_{L^\infty} \norm{f}_{H^s}\Big)^2\,.
\end{align*}

The previous statement extends to vector-valued functions $f$. In particular, it holds true if we replace $f$ by any $d$-dimensional vector field
$u\in H^s(\T^d)$ and if we take $P(\d)\,=\,\D$ to be the symmetric part of the Jacobian matrix of $u$.
\end{prop}

\begin{proof}
The proof of the previous proposition is based on the dyadic characterisation of Sobolev spaces, the equivalence of norms \eqref{eq:LP-Sob} and
an application of Lemma \ref{l:comm}.

Indeed, by relation \eqref{eq:LP-Sob}, one has
\[
\left\|\alpha\,P(\d)f\right\|_{H^{s}}^2\,\sim\,\left(\frac{1}{\left|\T^d\right|} \int_{\T^d} \alpha \,P(\d) f   \, \dx   \right)^2\,+\,
\sum_{j\in\Z}2^{2js}\,\left\|\Delta_j\big(\alpha\,P(\d)f\big)\right\|_{L^2}^2\,.
\]
At this point we observe that, for any $j\in\Z$ we can write
\[
 \left\|\Delta_j\big(\alpha\,P(\d)f\big)\right\|_{L^2}^2\,=\,
 \int_{\T^d}\alpha^2\,\left|\Delta_jP(\d)f\right|^2\,\dx\,+\,\left\|\comm{\Delta_j}{\alpha}\,P(\d)f\right\|_{L^2}^2\,.
\]

Therefore, by use of the inequality
\[
 \left(\frac{1}{\left|\T^d\right|} \int_{\T^d} \alpha \,P(\d) f   \, \dx   \right)^2\,\lesssim\,\int_{\T^d}\alpha^2\,\left|P(\d)f\right|^2\,\dx\,\lesssim\,
\left\|\nabla f\right\|_{L^\infty}^2\,\left\|\alpha\right\|_{L^2}^2\,,
\]
and of Lemma \ref{l:comm}, or better of estimate \eqref{est:comm3}, we can conclude the proof.
\end{proof}

Before concluding this part, we still need some non-linear estimates in $H^s(\T^d)$. The first one concerns the product of two functions
and is a classical property.
\begin{lemma}\label{l:prod}
Given $s>0$, the space $L^\infty(\T^d) \cap H^s(\T^d)$ is a Banach algebra. In addition, a constant $C=C(s)>0$ exists such that,
for any $u,v\,\in\,L^\infty(\T^d)\cap H^s(\T^d)$, one has
\[
\norm{u\,v}_{H^s}\, \leq\, C\, \Big( \norm{u}_{L^\infty}\, \norm{v}_{H^s} \,+\, \norm{u}_{H^s}\, \norm{v}_{L^\infty}  \Big)\,.
\]

In addition, the same estimate holds true even when replacing the $H^s$ norm with its homogeneous couterpart $\dot H^s$.
\end{lemma} 
The proof goes along the main lines of Corollary 2.86 of \cite{B-C-D} (where the property is established in the $\R^d$ setting). 
In particular, it is based on paraproduct decomposition. It is easy to see that everything can be transposed to the geometry of the torus,
up to defining, for any $N\in\Z$, the low frequency cut-off operators
\[
 S_Nu\,:=\,\oline u\,+\,\sum_{j\leq N-1}\Delta_ju\,.
\]
We also refer to Section 2 of \cite{Danchin} for more details.

\medbreak
The second non-linear estimate which we need is about left composition of $H^s(\T^d)$-functions $\o$ by smooth functions $F$.
However, we are in a situation where we cannot apply the classical paralinearisation results (for which we refer to \tsl{e.g.} Section 2.8 of \cite{B-C-D}).

As a matter of fact, in view of applications to the study of well-posedness of equations \eqref{eq:kolm_d}, we need to consider
the case in which $F(\t)\longrightarrow+\infty$ for $\t\to0^+$, whereas $\o\geq \o_*>0$ is uniformly bounded away from $0$,
but its infimum $\o_*$ is in fact time-dependent and approaches $0$ when the time increases. We refer to Subsection \ref{ss:energy} for more details.

As a consequence, we need to track the precise dependence of all the estimates on the value of $\omega_*\,=\,\inf\o$.
As this of course heavily depends on the function $F$, we will do so only for a special choice of such $F$, which is relevant for applications
to the study of the well-posedness of system \eqref{eq:kolm_d}. On the other hand, we will exploit the fact that,
for integer values of the regularity index $s\in\N$, one has precise computations which easily allow to track the dependence on $\o_*$. Therefore,
for general $s>1+d/2$, we need to pass to integer\footnote{Throughout this text, we note by $[s]$ the integer part
of a real number $s\in\R$, namely the biggest integer which is lower than, or equal to, $s$.} regularities $[s]$, thus losing
some derivatives in the estimates.

\begin{lemma}\label{l:comp}
Let $d\geq1$ and $s>1+d/2$. Take a positive function $\o\in H^s(\T^d)$ and define $\o_o\,:=\,\inf_{x\in\T^d}\o(x)$. Assume that $\o_o>0$.

Then the function $F(\o)\,:=\,1/\sqrt{\o}$ belongs to $H^s(\T^d)$. In addition, there exists
a ``universal'' constant $C\,=\,C(s)\,>\,0$, only depending on the value of the regularity index $s$, such that the following estimate holds true:
\[
\left\|\frac{1}{\sqrt{\o}}\right\|_{H^s}\,\leq\,C\;\frac{1\,+\,(\o_o)^{1+[s]}}{(\o_o)^{\frac{3}{2}+[s]}}\;
\left(1\,+\,\left\|\nabla\o\right\|_{L^\infty}^{[s]}\right)\,\left\|\o\right\|_{H^s}\,.
\]
\end{lemma}

\begin{proof}
In order to give a precise dependence of the estimates on $\o_o$, we need to exploit the explicit computations which are available
in the case of integer regularity indices $n\in\N$. This prompts us to use \eqref{eq:Hs} and write
\begin{align} \label{est:sqrt-o}
\left\|\frac{1}{\sqrt{\o}}\right\|_{H^s}^2\,&\lesssim\,\left\|\frac{1}{\sqrt{\o}}\right\|_{L^2}^2\,+\,
\left\|\nabla\left(\frac{1}{\sqrt{\o}}\right)\right\|_{\dot H^{s-1}}^2\,\lesssim\,
\frac{1}{\o_o}\,+\,\left\|\frac{1}{\o^{3/2}}\,\nabla\o\right\|_{\dot H^{s-1}}^2\,.
\end{align}
Now, thanks to Lemma \ref{l:prod}, we can bound
\begin{align*}
\left\|\frac{1}{\o^{3/2}}\,\nabla\o\right\|_{\dot H^{s-1}}^2\,&\lesssim\,
\left(\left\|\frac{1}{\o^{3/2}}\right\|_{L^\infty}\,\left\|\nabla\o\right\|_{\dot H^{s-1}}\,+\,
\left\|\frac{1}{\o^{3/2}}\right\|_{\dot H^{s-1}}\,\left\|\nabla\o\right\|_{L^\infty}\right)^2 \\
&\lesssim\,\left(\frac{1}{\o_o^{3}}\,\left\|\o\right\|_{H^s}^2\,+\,\left\|\frac{1}{\o^{3/2}}\right\|^2_{\dot H^{s-1}}\,\left\|\nabla\o\right\|_{L^\infty}^2\right)\,.
\end{align*}

Assume that $s-1\geq1$ for a while. Then, in order to estimate $\o^{-3/2}$ in $\dot H^{s-1}$, we can proceed in the same way. More precisely,
we write
\begin{align*}
\left\|\frac{1}{\o^{3/2}}\right\|^2_{\dot H^{s-1}}\,&\sim\,\left\|\nabla\left(\frac{1}{\o^{3/2}}\right)\right\|^2_{\dot H^{s-2}}\,=\,
\left\|\frac{1}{\o^{5/2}}\,\nabla\o\right\|^2_{\dot H^{s-2}}\,,
\end{align*}
which implies, together with Lemma \ref{l:prod} again, the estimate
\begin{align*}
\left\|\frac{1}{\o^{3/2}}\right\|^2_{\dot H^{s-1}}
&\lesssim\,\left(\left\|\frac{1}{\o^{5/2}}\right\|_{L^\infty}\,\left\|\nabla\o\right\|_{\dot H^{s-2}}\,+\,
\left\|\frac{1}{\o^{5/2}}\right\|_{\dot H^{s-2}}\,\left\|\nabla\o\right\|_{L^\infty}\right)^2 \\
&\lesssim\,\left(\frac{1}{\o_o^{5}}\,\left\|\o\right\|_{H^s}^2\,+\,\left\|\frac{1}{\o^{5/2}}\right\|^2_{\dot H^{s-2}}\,\left\|\nabla\o\right\|_{L^\infty}^2\right)\,.
\end{align*}
Iterating this argument $[s]$ times, and inserting the resulting expressions into \eqref{est:sqrt-o}, we find
\begin{align}
\label{est:sqrt-o_2}
\left\|\frac{1}{\sqrt{\o}}\right\|_{H^s}\,&\lesssim\,\frac{1}{\sqrt{\o_o}}\,+\,\frac{1}{\o_o^{3/2}}\,\|\o\|_{H^s}\,+\,
\frac{1}{\o_o^{5/2}}\,\|\o\|_{H^s}\,\left\|\nabla\o\right\|_{L^\infty} \\
\nonumber
&\qquad\qquad +\,\ldots\,+\,\frac{1}{\o_o^{\frac{1}{2}+[s]}}\,\|\o\|_{H^s}\,\left\|\nabla \o\right\|_{L^\infty}^{[s]-1}\,+\,
\left\|\frac{1}{\o^{\frac{1}{2}+[s]}}\right\|_{\dot H^{s-[s]}}\,\left\|\nabla \o\right\|_{L^\infty}^{[s]}\,.
\end{align}

As a last step, we take advantage of the fact that $0\leq s-[s]<1$. If $s-[s]=0$, we can bound
\[
\left\|\frac{1}{\o^{\frac{1}{2}+[s]}}\right\|_{\dot H^{s-[s]}}\,\lesssim\,\frac{1}{\o_o^{\frac{1}{2}+[s]+1}}\left\|\o\right\|_{L^2}\,,
\]
whereas in the case $s-[s]>0$ we rather compute
\begin{align*}
 \left\|\frac{1}{\o^{\frac{1}{2}+[s]}}\right\|_{\dot H^{s-[s]}}\,\lesssim\,\left\|\frac{1}{\o^{\frac{1}{2}+[s]}}\right\|_{H^1}\,\lesssim\,
\frac{1}{\o_o^{\frac{1}{2}+[s]}}\,+\,\frac{1}{\o_o^{\frac{1}{2}+[s]+1}}\,\left\|\nabla\o\right\|_{L^2}\,.
\end{align*}
Then, inserting this last bound into \eqref{est:sqrt-o_2} and observing that, for any real number $a>0$, one has
\[
\sum_{n=0}^{[s]+1}\frac{1}{a^{\frac{1}{2}+n}}\,\lesssim\,\frac{1}{a^{\frac{1}{2}}}\,+\,\frac{1}{a^{n+\frac{3}{2}}}\,=\,\frac{1+a^{n+1}}{a^{n+\frac{3}{2}}}\,,
\]
we finally deduce the sought estimate.
\end{proof}

\section{\textsl{A priori} estimates} \label{s:a-priori}

The goal of this section is to establish \tsl{a priori} estimates for smooth solutions to system \eqref{eq:kolm_d}.
Similarly to \cite{F-GB}, their derivation is based on a two-step
procedure: first of all, we bound the low regularity norms using the parabolic maximum principle and basic energy estimates; after that, we
use the Littlewood-Paley machinery to derive bounds for the higher regularity norms. All together, those estimates will imply the sought control
of the Sobolev norm $H^s$ of the solution.

We point out that, in order to carry out the higher order estimates, it will be fundamental to resort to the formulation of the system,
pointed out in \cite{F-GB}, in the new unknowns $\big(u,\o,\bt\big)$, where we have set $\bt\,:=\,\sqrt{k}$.

\subsection{Bounds for the low regularity norms} \label{ss:energy}

Here we derive \tsl{a priori} estimates for the low regularity norms of a (supposed to exist) smooth solution $\big(u,\o,k\big)$ of \eqref{eq:kolm_d}.

First of all, we notice that, using the parabolic structure of the equations,
we can derive pointwise lower and upper bounds for the functions $\o$ and $k$. Let us define the quantities 
\begin{equation*}
	\omega_\ast := \min_{x \in \Omega} \o_0(x)\,, \qquad\quad  \omega^\ast := \max_{x \in \Omega} \o_0(x)\,, \qquad\quad k_\ast := \min_{x \in \Omega} k_0(x)\,,
\end{equation*}
where we have $\o_*\geq0$ and $k_*\geq0$.
Then, arguing as in \cite{F-GB} (see also \cite{Mie-Nau}, \cite{Bul-Mal}) allows us to get the following bounds: 
\begin{equation}\label{eq:obound}
	\forall \, (t,x) \in \R_+ \times\Omega\,, \qquad 0\, <\, \o_{\min} (t)\, \leq\, \o(t,x)\, \leq\,\o^{\max}(t)\, \leq\, \o^\ast\,,
\end{equation}
where we have defined
\[
\o_{\min} (t)\,:=\,\frac{ \o_\ast}{\o_\ast \alpha_2 t + 1}\qquad\qquad \mbox{ and }\qquad\qquad
\o^{\max}(t)\,:=\,\frac{\o^\ast}{\o^\ast \alpha_2 t + 1}\,,
\]
and also
\begin{equation} \label{eq:kbound}
\forall \, (t,x) \in \R_+ \times \Omega\,, \qquad k(t,x)\,\geq\,k_{\min}(t)\,:=\,\frac{ k_\ast}{(\o^\ast \alpha_2 + 1)^{1/\alpha_2}}\,\geq\,0\,.
\end{equation}
In particular we deduce that, if $k_*=0$, then $k(t,x)\geq 0$ at any time $t\geq0$ and for any $x\in\Omega$.

\medbreak
Next, we perform energy estimates. To begin with, we observe that a simple energy method for the equations for $u$ yields the identity
\begin{equation*}
\frac{1}{2}\,	\frac{d}{dt} \int_{\Omega} \,|u|^2\, \dx\, +\, \nu\, \int_{\Omega}\, \frac{k}{\o}\, |\D u|^2\, \dx\, =\, 0\,,
\end{equation*}
where we have used also the $L^2$ orthogonality between $u$ and $\nabla\pi$, owing to the divergence-free condition $\div u=0$.
Integrating in time the previous relation, we find that 
\begin{equation}\label{est:ul2}
\forall\,t\geq0\,,\qquad\qquad
\norm{u(t)}^2_{L^2}\, +\, 2\, \nu\, \int_0^t \int_{\Omega}\, \frac{k}{\o}\, |\D u|^2 \dx\dd\t\, \leq\, \norm{u_0}_{L^2}^2\,.
\end{equation}

Performing similar computations on the (scalar) equation for $\o$, we get
\begin{equation*}
	\frac{1}{2}\,\frac{d}{dt} \int_{\Omega}\, |\o|^2\, \dx\, +\, \alpha_1 \int_{\Omega}\, \frac{k}{\o}\, |\nabla \o|^2\, \dx\,+\,
\alpha_2 \int_{\Omega}\, \o^3\, \dx\, =\, 0\,.
\end{equation*}
After an integration in time, we deduce that
\begin{equation}\label{est:ol2}
\forall\,t\geq0\,,\qquad\quad
\norm{\o(t)}^2_{L^2}\, +\, 2\, \alpha_1 \int_0^t \int_{\Omega}\, \frac{k}{\o}\, |\nabla \o |^2\, \dx\dd\t\, +\,
2 \,\alpha_2 \int_0^t \int_{\Omega}\, \o^3\, \dx\dd\t\, \leq\,  \norm{\o_0}_{L^2}^2\,.
\end{equation}

Unfortunately, the same computations have no chance to work for the last unknown $k$, owing to the presence in its equation of the $\alpha_4$ term,
which is merely $L^1_{t,x}$ (keep in mind \eqref{est:ul2} above).
Instead, we perform a simple integration of the equation over $\Omega$, getting in this way
\begin{equation*}
\frac{d}{dt} \int_{\Omega}\, k\, \dx\, +\, \int_{\Omega}\, k\, \o \,\dx\, =\, \alpha_4\, \int_{\Omega}\, \frac{k}{\o}\, \seminorm{\D u}^2 \dx\,.
\end{equation*}
Integrating in time the previous relation and using \eqref{est:ul2}, we find
\begin{equation}\label{est:kl1}
\forall\,t\geq0\,,\qquad\qquad
\norm{k(t)}_{L^1}\, +\, \int_0^t \int_{\Omega}\, k\, \o\, \dx \dd\t\, \leq\, \frac{\alpha_4}{2\nu}\, \norm{u_0}_{L^2}^2\, +\,  \norm{k_0}_{L^1}\,.
\end{equation}
The discussion in \cite{F-GB} suggests to introduce the ``good unknown'' $\beta := \sqrt{k}$. Thus, the previous estimate translates
into a $L^2$ control for variable $\beta$, namely
\begin{equation}\label{est:bl2}
\forall\,t\geq0\,,\qquad\qquad
\norm{\beta(t)}_{L^2}^2\, +\, \int_0^t \int_{\Omega} \,\beta^2\, \o \,\dx\dd\t\, \leq \, \frac{\alpha_4}{2\nu}\, \norm{u_0}_{L^2}^2\, + \, \norm{\beta_0}_{L^2}^2\,.
\end{equation}

\subsection{Reformulation of the system and localisation} \label{ss:reform}

After having established estimates for the low regularity norms (\tsl{i.e.} for low frequencies), we need to control the high regularity norms, namely
the high frequencies of the solution. However, before doing that, some preparation is needed.

To begin with, in order to deal with the degeneracy of the system when $k\approx0$, inspired by \cite{F-GB} we resort to the new unknown
\[
\bt\,:=\,\sqrt{k}
\]
introduced above, keep in mind \eqref{est:bl2}.
In particular, propagation of high regularity norms for $k$ will be done through propagation of high regularity for $\bt$.

Observe that, by (formally) multiplying the third equation in \eqref{eq:kolm_d} by $1/(2\,\sqrt{k})$, we easily derive the equation satisfied by $\beta$:
\begin{equation*}
  \d_t\bt\,+\,u\cdot\nabla \bt \,-\,\alpha_3\,\div\left(\dfrac{\beta^2}{\o}\,\nabla \bt \right)\,=\,-\frac{\,\bt \,\o}{2}\,+\,\dfrac{\alpha_4}{2}\,\dfrac{\bt}{ \o}\,\big|\D u\big|^2 + \alpha_3 \frac{\beta}{\o} \seminorm{\nabla \beta}^2. 
\end{equation*}
Thus, we can recast system \eqref{eq:kolm_d} as a system for the new triplet of unknowns $(u,\o,\bt)$: we get
\begin{equation} \label{eq:kolm_dnew}
	\left\{\begin{array}{l}
    \d_tu\,+\,(u\cdot\nabla) u\,+\,\nabla\pi\,-\,\nu\,\div\left(\dfrac{\bt^2}{\o}\,\D u\right)\,=\,0 \\[1ex]
    \d_t\o\,+\,u\cdot\nabla\o\,-\,\alpha_1\,\div\left(\dfrac{\bt^2}{\o}\,\nabla\o\right)\,=\,-\,\alpha_2\,\o^2 \\[1ex]
    \d_t\bt\,+\,u\cdot\nabla \bt \,-\,\alpha_3\,\div\left(\dfrac{\beta^2}{\o}\,\nabla \bt \right)\,=\,-\dfrac{\,\bt \,\o}{2}\,+\,\dfrac{\alpha_4}{2}\,\dfrac{\bt}{ \o}\,\big|\D u\big|^2 \,+ \, \alpha_3 \, \dfrac{\beta}{\o} \, \seminorm{\nabla \beta}^2 \\[2ex]
    \div u\,=\,0\,.
	\end{array}
	\right.
\end{equation}
Our next goal is to perform $H^s$ estimates on this new system.
This can be done in a classical way, by taking advantage of the characterisation \eqref{eq:Hs} of Sobolev spaces in terms of Littlewood-Paley decomposition.
As a matter of fact, we notice that, owing to the bounds established in Subsection \ref{ss:energy}, only $\dot H^s$ estimates are needed.

In order to tackle $\dot H^s$ estimates, the first step consists in localising the equations in frequencies \tsl{via} the operators $\Delta_j$.
Of course, this procedure will create some commutators. 
Indeed, by applying the operator $\Delta_j$ to each equation appearing in \eqref{eq:kolm_dnew}, standard computations yield
\begin{align*}
\big(\d_t\,+\,u\cdot\nabla\big) \Delta_j u\,+\,  \Delta_j\nabla \pi\,-\,\nu\,\div\left(\dfrac{\bt^2}{\o}\,\D \Delta_j u\right)\, &= \,\mds C_{u,j}^1\,+\,
\nu\,\mds C_{u,j}^2\\
\big(\d_t\,+\,u\cdot\nabla\big)\Delta_j \o\,-\,\alpha_1\,\div\left(\dfrac{\beta^2}{\o}\,\nabla \Delta_j \o\right)\,&=\,\mds C^1_{\o,j}\,+\,
\alpha_1\,\mds C^2_{\o,j}\,-\,\alpha_2\,\Delta_j\left(\o^2 \right) \\
\big(\d_t\,+\,u\cdot\nabla\big)\Delta_j \bt \,-\,\alpha_3\,\div\left(\dfrac{\beta^2}{\o}\,\nabla \Delta_j  \bt \right)\,&=\,\mds C^1_{\bt,j}\,+\,
\alpha_3\,\mds C^2_{\bt,j}\,-\,\dfrac{ \Delta_j (\bt \,\o)}{2} \\
&\quad +\,\frac{\alpha_4}{2}\, \Delta_j \left(\dfrac{\bt}{ \o}\,\big|\D u\big|^2\right)\, +\,
\alpha_3\, \Delta_j  \left(\dfrac{\beta}{\o} \seminorm{\nabla \beta}^2 \right)\,,
\end{align*}
where, for $f\in\{u,\o,\bt\}$, we have defined the commutator terms
\begin{align*}
\mds C^1_{f,j}\,&:=\, \left[u, \Delta_j \right]\cdot \nabla f 
\qquad\qquad \mbox{ and } \qquad\qquad
\mds C^2_{f,j}\,:=\,\div  \left( \comm{ \Delta_j}{\frac{\beta^2}{\o}}\, \nabla f    \right)\,, 
\end{align*}
with the convention that, when $f=u$, one has to change $\nabla u$ into $\D u$ in the definition of $\mds C^2_{u,j}$.

\medbreak
The goal of the next subsection is to perform energy estimates on the previous localised equations. As we will see, the main problems
will come from the analysis of the commutator terms. Observe that, while bounding the terms $\mds C^1_{f,j}$ is somehow classical,
the estimate for the $\mds C^2_{f,j}$ will be much more involved, due to the degeneracy of the system for $k\approx0$ (that is, for $\bt\approx0$).

\subsection{Estimates for the localised system} \label{ss:est-local}

We are ready to tackle energy estimates for the localised equations written above.
Thanks to the Littlewood-Paley characterisation of Sobolev spaces and to the estimates of Subsection \ref{ss:energy},
it is enough to bound the homogenous part of the Sobolev norm, namely
$$ \norm{f}_{\dot{H}^s} \,\sim\, \sum_{j \in \Z} 2^{2js}\, \norm{\Delta_j f}_{L^2}^2\,.$$
Recall that this sum reduces in fact to a sum for $j\geq -N$, for some lage enough $N\in\N$.

However, before performing estimates,, let us introduce some convenient notation.
In what follows, we generally use the notation $f \lesssim g$ to denote that there exists a multiplicative constant $c>0$,
only depending on the parameters $\big(d,s,\nu,\alpha_1,\alpha_2,\alpha_3,\alpha_4,\o_\ast, \o^\ast\big)$ of the system, such that
$ f\, \leq\, c\, g$.
In addition, for the sake of simplicity and when it does not cause any ambiguity, we will drop the time dependence from the notation through the estimates.

\subsubsection{Energy estimates for the dyadic blocks} \label{sss:dyadic}

We start by considering the equation for $\Delta_ju$. Performing an energy estimate for this quantity, owing to the divergence-free condition over $u$, we find
\begin{align*}
\frac{1}{2}\, \frac{\dd}{\dt} \norm{\Delta_j u }_{L^2}^2 \, + \,  \nu\, \int_{\Omega} \frac{\bt^2}{\o}\, \seminorm{\D \Delta_j u}^2 \,  \dx 
\, &= \,  \int_{\Omega} \mds C^1_{u,j}\cdot \Delta_j u \, \dx \, + \, \nu \, \int_{\Omega}  \mds C^2_{u,j}\cdot\Delta_j u \, \dx \,.
\end{align*}
By multiplying the previous equation by $2^{2js}$ and summing over the integers $j \in \Z$, we infer the following identity:
\begin{align}
 \label{eq:highu}
&\frac{1}{2}\, \frac{\dd}{\dt}\sum_{j \in \Z} 2^{2js}\, \norm{\Delta_j u }_{L^2}^2 \, + \,
\nu\,\sum_{j \in \Z} 2^{2js}\, \norm{ \frac{\bt}{\sqrt \o}\, \D (\Delta_j u)  }_{L^2}^2 \\
\nonumber
&\qquad\qquad\qquad\qquad\qquad
=  \, \sum_{j \in \Z} 2^{2js} \int_{\Omega} \mds C^1_{u,j}\cdot \Delta_j u \, \dx \,+ \,
\nu \,  \sum_{j \in \Z} 2^{2js} \int_{\Omega}  \mds C^2_{u,j}\cdot \Delta_j u \, \dx \,.
\end{align}
It is apparent that we need to control the commutator terms appearing in the right-hand side of the previous relation. This will be done in Paragraph
\ref{sss:commutators} below. For the time being, let us perform similar computations on the equations for $\Delta_j\o$ and $\Delta_j\bt$. 

\medbreak
So, let us consider the equation for $\Delta_j\o$.
Similar computations as above yield
\begin{align*}
&\frac{1}{2}\,\frac{\dd}{\dt}\sum_{j \in \Z} 2^{2js}\,\norm{\Delta_j \o}_{L^2}^2 \,+ \,
\alpha_1\, \sum_{j \in \Z} 2^{2js}\, \norm{\frac{\bt}{\sqrt \o}\, \nabla\Delta_j \o}_{L^2}^2\, \\
\nonumber
&\quad
=\,\sum_{j \in \Z} 2^{2js} \int_{\Omega} \mds C^1_{\o,j} \Delta_j \o\, \dx  \,
+\,\alpha_1 \, \sum_{j \in \Z} 2^{2js} \int_{\Omega}\mds C^2_{\o,j} \Delta_j \o\, \dx\,-\,
\alpha_2\, \sum_{j \in \Z} 2^{2js} \int_{\Omega} \Delta_j \left(\o^2\right) \,\Delta_j \o \,\dx\, .
\end{align*}
where we have used once again the fact that $\div u=0$.

Let us leave the commutator terms on a side for a while and rather focus on the last term appearing in the right-hand side of the previous relation.
This term can be easily controlled thanks to the Cauchy-Schwarz inequality in the following way:
\begin{align} 
\label{est:D_jo^2}
\seminorm{\sum_{j \in \Z} 2^{2js} \int_{\Omega} \Delta_j\left(\o^2\right)\, \Delta_j \o\, \dx }\,&\lesssim\,
\left(\sum_{j\in\Z}2^{2js}\,\norm{\Delta_j\left(\o^2\right)}_{L^2}^2\right)^{1/2}\,\left(\sum_{j\in\Z}2^{2js}\,\norm{\Delta_j\o}_{L^2}^2\right)^{1/2} \\
\nonumber
&\lesssim\,\norm{\o^2}_{\dot H^s}\,\norm{\o}_{\dot H^s}\;\lesssim\;\norm{\o}_{L^\infty}\,\norm{\o}_{\dot H^s}^2\,,
\end{align}
where we have used also Lemma \ref{l:prod}.
Inserting this bound into the previous relation, we get
\begin{align}
\label{eq:higho}
&\frac{1}{2}\,\frac{\dd}{\dt}\sum_{j \in \Z} 2^{2js}\,\norm{\Delta_j \o}_{L^2}^2 \,+ \,
\alpha_1\, \sum_{j \in \Z} 2^{2js}\, \norm{\frac{\bt}{\sqrt \o}\, \nabla\Delta_j \o}_{L^2}^2\, \\
\nonumber
&\quad\qquad\qquad 
\lesssim\,\sum_{j \in \Z} 2^{2js} \int_{\Omega} \mds C^1_{\o,j}\, \Delta_j \o\, \dx  \,
+\,\alpha_1 \, \sum_{j \in \Z} 2^{2js} \int_{\Omega}\mds C^2_{\o,j}\, \Delta_j \o\, \dx\,+\,\norm{\o}_{L^\infty}\,\norm{\o}_{\dot H^s}^2\,.
\end{align}
As before, we postpone the control of the commutator terms to the next paragraph.

\medbreak
Finally, let us consider the equation for $\Delta_j\bt$.
Testing it against $\Delta_j \bt$ itself and integrating over $\Omega$, we obtain
\begin{align}
\label{est:D_jb_partial}
&\frac{1}{2}\,\frac{\dd }{\dt}\sum_{j \in \Z} 2^{2js}\, \norm{\Delta_j \bt}_{L^2}^2\,+\,
\alpha_3\,\sum_{j \in \Z} 2^{2js} \,\norm{\frac{\bt}{\sqrt \o} \, \nabla\Delta_j \beta}_{L^2}^2 \\
\nonumber
& = \sum_{j \in \Z} 2^{2js}\int_{\Omega}\mds C^1_{\bt,j}\, \Delta_j \bt \,\dx \,+\,
\alpha_3 \,\sum_{j \in \Z} 2^{2js} \int_{\Omega}\mds C^2_{\bt,j} \, \Delta_j \bt \,\dx \,-\,
\dfrac{1}{2}\sum_{j \in \Z} 2^{2js} \int_{\Omega}\Delta_j\big(\bt \,\o\big)\, \Delta_j \bt\, \dx   \\
\nonumber
&\qquad\quad +\,\frac{\alpha_4}{2}\,\sum_{j \in \Z} 2^{2js} \int_{\Omega} \Delta_j\left(\dfrac{\bt}{ \o}\,\big|\D u\big|^2\right)\,\Delta_j \bt\, \dx\,+\,
\alpha_3  \sum_{j \in \Z} 2^{2js} \int_{\Omega}\Delta_j\left(\dfrac{\beta}{\o} \seminorm{\nabla \beta}^2 \right)\, \Delta_j \bt\, \dx\,.
\end{align}

Notice that, repeating \tsl{mutatis mutandis} the computations leading to \eqref{est:D_jo^2}, we can estimate
\begin{align*}
\left|\sum_{j \in \Z} 2^{2js} \int_{\Omega}\Delta_j\big(\bt \,\o\big)\, \Delta_j \bt\, \dx\right|\,&\lesssim\,
\left\|\bt\,\o\right\|_{\dot H^s}\,\left\|\bt\right\|_{\dot H^s}\,\lesssim\,
\left\|\big(\o,\bt\big)\right\|_{L^\infty}\,\left\|\big(\o,\bt\big)\right\|_{\dot H^s}^2\,.
\end{align*}
However, the same argument has no chance to work when applied to the terms appearing in the last line of \eqref{est:D_jb_partial},
because $\nabla \bt$ and $\D u$ do not belong to $H^s$, but only to $H^{s-1}$.

In order to avoid the previously mentioned loss of derivative, the idea is to take advantage of the coefficient $\bt/\sqrt{\o}$ appearing
in front of the bad terms and of the (degenerate) parabolic smoothing of the equations. More precisely, we notice that we can write
\[
\sum_{j \in \Z} 2^{2js} \int_{\Omega} \Delta_j \left(\dfrac{\bt}{ \o}\,\big|\D u\big|^2\right)\, \Delta_j \bt\, \dx\,  =\,
\sum_{j \in \Z} 2^{2js} \int_{\Omega} \frac{\bt}{\o}\, \Delta_j\D u :\D u \, \Delta_j \bt\, \dx  \,+\,
\mds C^3[u,\bt]\,,
\]
where we have defined the new commutator term
\begin{equation*}
\mds C^3[u,\bt]\,:=\,\sum_{j \in \Z} 2^{2js} \int_{\Omega}\comm{\Delta_j}{\frac{\bt}{\o}\,\D u} : \D u \, \Delta_j \bt \,\dx\,.
\end{equation*}
At this point, it is easy to bound the former sum in the right-hand side as
\begin{align*}
\left|\sum_{j \in \Z} 2^{2js} \int_{\Omega} \frac{\bt}{\o}\, \Delta_j\D u :\D u \, \Delta_j \bt\, \dx\right|\,&\lesssim\,
\frac{1}{\sqrt{\omega_{\min}}}\,\left\|\nabla u\right\|_{L^\infty} \\
&\qquad \times\,\left(\sum_{j \in \Z} 2^{2js}\,\left\|\frac{\bt}{\sqrt{\o}}\, \Delta_j\D u\right\|_{L^2}^2\right)^{1/2}
\,\left(\sum_{j \in \Z} 2^{2js}\,\left\|\Delta_j\bt\right\|_{L^2}^2\right)^{1/2} \\
&\lesssim\,\frac{1}{\sqrt{\omega_{\min}}}\,\left\|\nabla u\right\|_{L^\infty}\,
\left(\sum_{j \in \Z} 2^{2js}\,\left\|\frac{\bt}{\sqrt{\o}}\, \Delta_j\D u\right\|_{L^2}^2\right)^{1/2}\,\left\|\bt\right\|_{\dot H^s}\,.
\end{align*}

From analogous computations, we deduce also that
\[
\sum_{j \in \Z} 2^{2js} \int_{\Omega}\Delta_j\left(\dfrac{\beta}{\o} \seminorm{\nabla \beta}^2 \right)\, \Delta_j \bt\, \dx\,=\,
\sum_{j \in \Z} 2^{2js} \int_{\Omega} \dfrac{\bt}{\o}\, \Delta_j\nabla\bt :\nabla \bt \, \Delta_j \bt\, \dx  \,+\,
\mds C^4[\bt,\bt]\,,
\]
where this time we have set
\begin{equation*}
\mds C^4[\bt,\bt]\,:=\,\sum_{j \in \Z} 2^{2js} \int_{\Omega}\comm{\Delta_j}{\frac{\bt}{\o}\,\nabla \bt} : \nabla\bt \, \Delta_j \bt \,\dx
\end{equation*}
and where the next estimate holds true:
\begin{align*}
\left|\sum_{j \in \Z} 2^{2js} \int_{\Omega} \frac{\bt}{\o}\, \Delta_j\nabla\bt :\nabla\bt \, \Delta_j \bt\, \dx\right|\,&\lesssim\,
\frac{1}{\sqrt{\omega_{\min}}}\,\left\|\nabla\bt\right\|_{L^\infty}\,
\left(\sum_{j \in \Z} 2^{2js}\,\left\|\frac{\bt}{\sqrt{\o}}\, \Delta_j\nabla\bt\right\|_{L^2}^2\right)^{1/2}\,\left\|\bt\right\|_{\dot H^s}\,.
\end{align*}

Inserting all those bounds into \eqref{est:D_jb_partial} and making use of the Young inequality, we find, for any $\de>0$ to be fixed later, the bound
\begin{align}
\label{est:D_jb}
&\frac{1}{2}\,\frac{\dd }{\dt}\sum_{j \in \Z} 2^{2js}\, \norm{\Delta_j \bt}_{L^2}^2\,+\,
\alpha_3\,\sum_{j \in \Z} 2^{2js} \,\norm{\frac{\bt}{\sqrt \o} \, \nabla\Delta_j \beta}_{L^2}^2 \\
\nonumber
&\leq\,\sum_{j \in \Z} 2^{2js}\int_{\Omega}\mds C^1_{\bt,j}\, \Delta_j \bt \,\dx \,+\,
\alpha_3 \,\sum_{j \in \Z} 2^{2js} \int_{\Omega}\mds C^2_{\bt,j} \, \Delta_j \bt \,\dx \,+\,
\mds C^3[u,\bt]\,+\,\mds C^4[\bt,\bt] \\
\nonumber
&\qquad +\,C\,\left(\left\|\big(\o,\bt\big)\right\|_{L^\infty}\,+\,\frac{1}{\o_{\min}}\,\left\|\big(\nabla u,\nabla\bt\big)\right\|_{L^\infty}^2\right)\,\left\|\big(\o,\bt\big)\right\|_{\dot H^s}^2 \\
\nonumber
&\qquad\qquad
+\,\de\,\sum_{j \in \Z} 2^{2js}\,\left\|\frac{\bt}{\sqrt{\o}}\, \Delta_j\D u\right\|_{L^2}^2\,+\,
\de\,\sum_{j \in \Z} 2^{2js}\,\left\|\frac{\bt}{\sqrt{\o}}\, \Delta_j\nabla\bt\right\|_{L^2}^2\,,
\end{align}
where the mulitiplicative constant $C>0$ depends also on $\de>0$.

\medbreak
It is time to sum up inequalities \eqref{eq:highu}, \eqref{eq:higho} and \eqref{est:D_jb}.
For simplicity of notation, let us introduce the (homogeneous) Sobolev energy of the solution,
\[
E_s(t)\,:=\,\left\|\big(u,\o,\bt\big)(t)\right\|_{\dot H^s}\,\sim\,\sum_{j \in \Z} 2^{2js}\, \norm{\Delta_ju(t)}_{L^2}^2\,+\,
\sum_{j \in \Z} 2^{2js}\, \norm{\Delta_j \o(t)}_{L^2}^2\,+\,\sum_{j \in \Z} 2^{2js}\, \norm{\Delta_j \bt(t)}_{L^2}^2\,,
\]
and the higher order energy
\[
F_s(t)\,:=\,\sum_{j \in \Z} 2^{2js} \,\norm{\frac{\bt}{\sqrt \o} \,\Delta_j\D u}_{L^2}^2\,+\,
\sum_{j \in \Z} 2^{2js} \,\norm{\frac{\bt}{\sqrt \o} \, \nabla\Delta_j \o}_{L^2}^2\,+\,
\sum_{j \in \Z} 2^{2js} \,\norm{\frac{\bt}{\sqrt \o} \, \nabla\Delta_j \beta}_{L^2}^2\,,
\]
coming from the (degenerate) parabolic effect.

Then, summing up \eqref{eq:highu}, \eqref{eq:higho} and \eqref{est:D_jb}, choosing $\de=\min\{\nu,\alpha_3\}/2$ in \eqref{est:D_jb} and finally
integrating in time, we infer, for any time $t\geq0$, the inequality
\begin{align}
\label{est:E_partial}
E_s(t)\,+\,\int^t_0F_s(\t)\,\dd\t\,&\lesssim\,E_s(0)\,+\,\int^t_0\left(\left\|\big(\o,\bt\big)\right\|_{L^\infty}\,+\,
\frac{1}{\o_{\min}}\,\left\|\big(\nabla u,\nabla\bt\big)\right\|_{L^\infty}^2\right)\,E_s(\t)\,\dd\t \\
\nonumber
&\quad +\,
\int^t_0\Big(\sum_{f\in\{u,\o,\bt\}}\left(\mds C^1[f,f]+\mds C^2[f,f]\right)+\mds C^3[u,\bt]+\mds C^4[\bt,\bt]\Big)\,\dd\t\,,
\end{align}
where $E_s(0)$ denotes the same quantity as $E_s$, but computed on the initial datum $\big(u_0,\o_0\,\sqrt{k_0}\big)$, and where,
for $\ell=1,2$ and $f\in\{u,\o,\bt\}$, we have set
\[
 \mds C^\ell[f,f]\,:=\,\sum_{j \in \Z} 2^{2js}\int_{\Omega}\mds C^\ell_{f,j}\, \Delta_jf \,\dx\,.
\]

At this point, adding estimates \eqref{est:ul2}, \eqref{est:ol2} and \eqref{est:bl2} to \eqref{est:E_partial}, we deduce a similar
inequality for the full $H^s$ energy of the solution, namely for
\[
 \mbf E_s(t)\,:=\,\left\|\big(u,\o,\bt\big)(t)\right\|_{H^s}\,\sim\,\norm{\big(u,\o,\bt\big)(t)}_{L^2}^2\,+\,E_s(t)\,.
\]
More precisely, we find, for any $t\geq0$, the bound
\begin{align}
\label{est:tot-E_partial}
\mbf E_s(t)\,+\,\int^t_0F_s(\t)\,\dd\t\,&\lesssim\,\mbf E_s(0)\,+\,\int^t_0\left(\left\|\big(\o,\bt\big)\right\|_{L^\infty}\,+\,
\frac{1}{\o_{\min}}\,\left\|\big(\nabla u,\nabla\bt\big)\right\|_{L^\infty}^2\right)\,\mbf E_s(\t)\,\dd\t \\
\nonumber
&\quad +\,
\int^t_0\Big(\sum_{f\in\{u,\o,\bt\}}\left(\mds C^1[f,f]+\mds C^2[f,f]\right)+\mds C^3[u,\bt]+\mds C^4[\bt,\bt]\Big)\,\dd\t\,,
\end{align}

\subsubsection{Commutator estimates} \label{sss:commutators}

In order to close the estimates, we need to control the commutator terms appearing in inequality \eqref{est:tot-E_partial}: this is the scope of the present paragraph.

We will start by bounding the terms of the form $\mds C^1[f,f]$, as their control is a direct application of Lemma \ref{l:comm}.
Then we will switch to the bounds for the terms $\mds C^3[u,\bt]$ and $\mds C^4[\bt,\bt]$, which are also based on Lemma \ref{l:comm}, but are slightly more involved.
As a matter of fact, the control of the Sobolev norm of the coefficient $1/\sqrt{\o}$ will cause some problems, due to the lower bound for $\o$, which,
as established in \eqref{eq:obound}, is not uniform in time. Finally, we will consider the terms of the type $\mds C^2[f,f]$,
whose bounds are more difficult to obtain and require a further decomposition.

\paragraph{Bounding the terms $\mds C^1[f,f]$ \\}

Given $f\in\{u,\o,\bt\}$, our first goal is to control
\begin{align*}
 \mds C^1[f,f]\,=\,\sum_{j \in \Z} 2^{2js}\int_{\Omega}\mds C^1_{f,j}\, \Delta_jf \,\dx\,,\qquad\qquad \mbox{ with }\qquad
 \mds C^1_{f,j}\,=\,\big[u,\Delta_j\big]\cdot\nabla f\,.
\end{align*}

A direct application of the Cauchy-Schwarz inequality and of Lemma \ref{l:comm} allows us to bound those terms as
\begin{align*}
\left|\mds C^1[f,f]\right|\,&\lesssim\,\left(\sum_{j \in \Z} 2^{2js}\,\left\|\mds C^1_{f,j}\right\|_{L^2}^2\right)^{1/2}\,
\left(\sum_{j \in \Z} 2^{2js}\,\norm{\Delta_jf}_{L^2}^2\right)^{1/2} \\
&\lesssim\,\Big(\|\nabla u\|_{L^\infty}\,\|f\|_{H^s}\,+\,\left\|\nabla f\right\|_{L^\infty}\,\left\|\nabla u\right\|_{H^{s-1}}\Big)\,\norm{f}_{\dot H^s} 
\end{align*}

This yields the control
\begin{equation} \label{est:C^1}
\sum_{f\in\{u,\o,\bt\}}\left|\mds C^1[f,f]\right|\,\lesssim\,\norm{\big(\nabla u,\nabla\o,\nabla\bt\big)}_{L^\infty}\,\mbf E_s(t)\,.
\end{equation}

\paragraph{Bounding the terms $\mds C^3[u,\bt]$ and $\mds C^4[\bt,\bt]$ \\}

We now consider the commutator terms $\mds C^3[u,\bt]$ and $\mds C^4[\bt,\bt]$. We recall here their definitions:
\begin{align*}
&\mds C^3[u,\bt]\,:=\,\sum_{j \in \Z} 2^{2js} \int_{\Omega}\comm{\Delta_j}{\frac{\bt}{\o}\,\D u} : \D u \, \Delta_j \bt \,\dx \\
&\qquad\qquad\qquad\qquad\qquad \mbox{ and }\qquad\qquad
\mds C^4[\bt,\bt]\,:=\,\sum_{j \in \Z} 2^{2js} \int_{\Omega}\comm{\Delta_j}{\frac{\bt}{\o}\,\nabla \bt} : \nabla\bt \, \Delta_j \bt \,\dx\,.
\end{align*}

Let us focus on $\mds C^3[u,\bt]$ first. Proceeding as above, thanks to the Yound inequality and Lemma \ref{l:comm} we find
\begin{align}
\label{est:C^3_part}
\left|\mds C^3[u,\bt]\right|\,&\lesssim\,\left(\sum_{j\in\Z}2^{2js}\,\norm{\comm{\Delta_j}{\frac{\bt}{\o}\,\D u}:\D u}_{L^2}^2\right)^{1/2}\,
\left(\sum_{j\in\Z}2^{2js}\,\norm{\Delta_j\bt}_{L^2}^2\right)^{1/2} \\
\nonumber
&\lesssim\,\left(\norm{\nabla\left(\frac{\bt}{\o}\,\D u\right)}_{L^\infty}\,\norm{u}_{H^s}\,+\,
\norm{\frac{\bt}{\o}\,\D u}_{H^s}\,\left\|\nabla u\right\|_{L^\infty}\right)\,\norm{\bt}_{\dot H^s}\,.
\end{align}
At this point, on the one hand we observe that
\begin{align*}
\norm{\nabla\left(\frac{\bt}{\o}\,\D u\right)}_{L^\infty}\,&\lesssim\,\frac{1}{\sqrt{\o_{\min}}}\,
\norm{\nabla\left(\frac{\bt}{\sqrt{\o}}\,\D u\right)}_{L^\infty}\,+\,\frac{1}{(\o_{\min})^{3/2}}\,\norm{\nabla\o}_{L^\infty}\,
\norm{\frac{\bt}{\sqrt{\o}}\,\D u}_{L^\infty} \\
&\lesssim\,\frac{1}{\sqrt{\o_{\min}}}\,\norm{\nabla\left(\frac{\bt}{\sqrt{\o}}\,\D u\right)}_{L^\infty}\,+\,
\frac{1}{(\o_{\min})^{2}}\,\norm{\bt}_{L^\infty}\,\norm{\big(\nabla u,\nabla\o\big)}^2_{L^\infty}
\end{align*}
and, on the other hand, we use product rules of Lemma \ref{l:prod} and Proposition \ref{p:key} to bound
\begin{align*}
\norm{\frac{\bt}{\o}\,\D u}_{H^s}\,&\lesssim\,\norm{\frac{1}{\sqrt{\o}}}_{L^\infty}\,\norm{\frac{\bt}{\sqrt{\o}}\,\D u}_{H^s}\,+\,
\norm{\frac{1}{\sqrt{\o}}}_{H^s}\,\norm{\frac{\bt}{\sqrt{\o}}\,\D u}_{L^\infty} \\
&\lesssim\,\frac{1}{\sqrt{\o_{\min}}}\,\left(\sum_{j\in\Z}2^{2js}\,\norm{\frac{\bt}{\sqrt{\o}}\,\Delta_j\D u}_{L^2}^2\right)^{1/2} \\
&\quad +\,
\frac{1}{\sqrt{\o_{\min}}}\,\left(\norm{\nabla u}_{L^\infty}\,\norm{\frac{\bt}{\sqrt{\o}}}_{H^s}\,+\,\norm{\frac{\bt}{\sqrt{\o}}}_{L^\infty}\,\norm{u}_{H^s}\,+\,
\norm{\bt}_{L^\infty}\,
\norm{\nabla u}_{L^\infty}\,\norm{\frac{1}{\sqrt{\o}}}_{H^s}\right)\,.
\end{align*}
Owing to Lemma \ref{l:comp}, we can further bound
\begin{align}
\label{est:1/o} 
 \norm{\frac{1}{\sqrt{\o}}}_{H^s}\,&\lesssim\,\frac{1+(\o_{\min})^{[s]+1}}{(\o_{\min})^{[s]+3/2}}\,\left(1+\norm{\nabla\o}_{L^\infty}^{[s]}\right)\,\norm{\o}_{H^s} \\
\label{est:b/o}
\norm{\frac{\bt}{\sqrt{\o}}}_{H^s}\,&\lesssim\,\frac{1}{\sqrt{\o_{\min}}}\,\norm{\bt}_{H^s}\,+\,
\norm{\bt}_{L^\infty}\,\frac{1+(\o_{\min})^{[s]+1}}{(\o_{\min})^{[s]+3/2}}\,\left(1+\norm{\nabla\o}_{L^\infty}^{[s]}\right)\,\norm{\o}_{H^s}\,.
\end{align}
In the end, putting all those inequalities into \eqref{est:C^3_part} and using that, by virtue of \eqref{eq:obound}, one has $\o_{\min}\approx (1+t)^{-1}$,
by careful computations we gather
\begin{align}
\label{est:C^3_fin}
\left|\mds C^3[u,\bt]\right|\,&\lesssim\,(1+t)^{1/2}\,\norm{\nabla\left(\frac{\bt}{\sqrt{\o}}\,\D u\right)}_{L^\infty}\,\mbf E_s\,+\,
(1+t)^{1/2}\,\norm{\nabla u}_{L^\infty}\,\sqrt{\mbf E_s}\,\sqrt{F_s}
\\
\nonumber
&\qquad +\,
(1+t)^{[s]+2}\,\left(1+\norm{\bt}_{L^\infty}\right)\,\left(1+\norm{\big(\nabla u,\nabla\o\big)}^{[s]+2}_{L^\infty}\right)\,\mbf E_s
\,.
\end{align}
Remark that the term
\[
 \norm{\nabla\left(\frac{\bt}{\sqrt{\o}}\,\D u\right)}_{L^\infty}
\]
cannot be decomposed further, as we do not know, at our level of regularity, whether $\nabla^2u$ is bounded or not in $L^\infty$.

In a completely analogous way, we also find
\begin{align}
\label{est:C^4_fin}
\left|\mds C^4[\bt,\bt]\right|\,&\lesssim\,(1+t)^{1/2}\,\norm{\nabla\left(\frac{\bt}{\sqrt{\o}}\,\nabla \bt\right)}_{L^\infty}\,\mbf E_s\,+\,
(1+t)^{1/2}\,\norm{\nabla\bt}_{L^\infty}\,\sqrt{\mbf E_s}\,\sqrt{F_s}
\\
\nonumber
&\qquad +\,
(1+t)^{[s]+2}\,\left(1+\norm{\bt}_{L^\infty}\right)\,\left(1+\norm{\big(\nabla \bt,\nabla\o\big)}^{[s]+2}_{L^\infty}\right)\,\mbf E_s
\,.
\end{align}

\paragraph{Bounding the terms $\mds C^2[f,f]$ \\}

As a last step, we bound the terms of the form $\mds C^2[f,f]$, for any given $f\in\{u,\o,\bt\}$. Recall that
\begin{align*}
 \mds C^2[f,f]\,=\,\sum_{j \in \Z} 2^{2js}\int_{\Omega}\mds C^2_{f,j}\, \Delta_jf \,\dx\,,\qquad\qquad \mbox{ with }\qquad
 \mds C^2_{f,j}\,=\,\div\left(\comm{\Delta_j}{\frac{\beta^2}{\o}}\,\nabla f\right)\,.
\end{align*}
Recall the convention we adopted above: when $f=u$, the term $\nabla f$ has to be replaced by $\D u$ in the definition of $\mds C^2_{u,j}$.

As it appears clear by \tsl{e.g.} performing an integration by parts inside the integral term, it is not possible to control
$\mds C^2[f,f]$ in a simple way by a direct use of Lemma \ref{l:comm}. This lack of control is essentially due to the degeneracy of the system when $\bt\approx0$.
In order to overcome this problem, the basic idea is to take advantage of the \emph{degenerate} parabolic smoothing, represented by the
term $F_s$ in the estimates. For doing so, we must ``distribute'' enough powers of the viscosity coefficient
$\bt/\sqrt{\o}$ to the terms presenting spatial derivatives: this requires to find a suitable commutator structure.

So, let us start performing careful computations in order to find the sought commutator structure.
To begin with, we observe that we can write
\begin{align}
\label{eq:C^2_dec-1}
\sum_{j \in \Z} 2^{2js}  \int_{\Omega}	\mds C^2_{f,j}\, \Delta_jf\,   \dx \,&=\,\mc T_1 \,+\,
\sum_{j \in \Z} 2^{2js}  \int_{\Omega}	\div  \left(   \comm{ \Delta_j}{\frac{\beta}{\sqrt \o}}\,\left(\frac{\bt}{\sqrt \o}\,\nabla f\right)\right)\,\Delta_jf\,\dx\,,
\end{align}
where we have defined
\begin{align*}
\mc T_1\,&:=\,\sum_{j \in \Z} 2^{2js}  \int_{\Omega}\div\left(\frac{\bt}{\sqrt{\o}}\,\comm{ \Delta_j}{\frac{\beta}{\sqrt \o}} \nabla f\right)\, \Delta_jf\,\dx\,.
\end{align*}
Observe that, after an integration by parts, the term $\mc T_1$ can be easily estimated by use of Lemma \ref{l:comm}. Indeed, thanks also
to the Cauchy-Schwarz inequalty, we have
\begin{align*}
\left|\mc T_1\right|\,&=\,
\left|\sum_{j \in \Z} 2^{2js}  \int_{\Omega}\frac{\bt}{\sqrt{\o}}\,\comm{ \Delta_j}{\frac{\beta}{\sqrt \o}} \nabla f\,\cdot\, \nabla \Delta_jf\,\dx\right| \\
&\lesssim\,\sum_{j \in \Z} 2^{2js}\,\norm{\frac{\bt}{\sqrt{\o}}\, \nabla\Delta_jf}_{L^2}\;\norm{ \comm{ \Delta_j}{\frac{\beta}{\sqrt \o}}\,\nabla f}_{L^2}
\,\lesssim\,\sqrt{F_s}\,\left(\sum_{j \in \Z} 2^{2js}\,\norm{ \comm{ \Delta_j}{\frac{\beta}{\sqrt \o}}\,\nabla f}^2_{L^2}\right)^{1/2} \\
&\lesssim\,\sqrt{F_s}\,\left(\norm{\nabla\left(\frac{\bt}{\sqrt{\o}}\right)}_{L^\infty}\,\norm{f}_{H^s}\,+\,
\norm{\nabla f}_{L^\infty}\,\norm{\frac{\bt}{\sqrt{\o}}}_{H^s}\right)^{1/2}
\end{align*}
Keeping in mind inequality \eqref{est:b/o} and the definition of $\o_{\min}$, in turn we get
\begin{align} \label{est:T_1}
\left|\mc T_1\right|\,&\lesssim\,\sqrt{F_s}\;\sqrt{\mbf E_s}\,(1+t)^{\frac{[s]}{2}+\frac{3}{4}}\,\left(1+\norm{\bt}_{L^\infty}\right)^{1/2}\,
\left(1+\norm{\big(\nabla u,\nabla\o,\nabla\bt\big)}_{L^\infty}^{[s]+1}\right)^{1/2}\,.
\end{align}

On the contrary, it is clear that a similar approach cannot work with the second term appearing on the right-hand side of equality \eqref{eq:C^2_dec-1}, because, after
an integration by parts, we would have no coefficients $\bt/\sqrt{\o}$ to put in front of $\nabla\Delta_jf$. Thus, the only solution is to explicitly
compute the divergence: we find
\begin{align}
\label{eq:C^2_dec-2}
&\sum_{j \in \Z} 2^{2js}  \int_{\Omega}	\div  \left(   \comm{ \Delta_j}{\frac{\beta}{\sqrt \o}}\,\left(\frac{\bt}{\sqrt \o}\,\nabla f\right)\right)\,\Delta_jf\,\dx \\
\nonumber
&\qquad\qquad\qquad\qquad\qquad \,=\,
\mc T_2\,+\,\sum_{j \in \Z}2^{2js}\int_{\Omega}\comm{ \Delta_j}{\nabla \left(\frac{\beta}{\sqrt \o} \right) }  \left(  \frac{\bt}{\sqrt{\o}} \nabla f \right)
\; \Delta_jf\,  \dx\,,
\end{align}
where this time we have set
\[
\mc T_2\,:=\,\sum_{j \in \Z} 2^{2js}\int_{\Omega}\left( \comm{ \Delta_j}{\frac{\beta}{\sqrt \o}} \div\left(  \frac{\bt}{\sqrt{\o}}\nabla f\right)\right)\Delta_jf\,\dx\,.
\]
The term $\mc T_2$ can be controlled in a quite direct way: by applying the Cauchy-Schwarz inequality and Lemma \ref{l:comm}, or rather of
inequality \eqref{est:comm3}, we infer that
\begin{align*}
\left|\mc T_2\right|\,&\lesssim\,
\left(\sum_{j \in \Z} 2^{2js}\,\norm{ \comm{ \Delta_j}{\frac{\beta}{\sqrt \o}}\,\div\left(\frac{\bt}{\sqrt{\o}}\,\nabla f \right)}_{L^2}^2\right)^{1/2}\,\left( \sum_{j \in \Z} 2^{2js}\,\norm{\Delta_jf  }_{L^2}^2 \right)^{1/2} \\
&\lesssim\,\left(  \norm{ \nabla \left( \frac{\bt}{\sqrt \o}\right)}_{L^\infty}\,\norm{\frac{\bt}{\sqrt \o}\,\nabla f }_{H^s}\,+\,
\norm{ \nabla \left(\frac{\bt}{\sqrt \o}\,\nabla f\right) }_{L^\infty } \,\norm{\frac{\bt}{\sqrt \o}}_{H^s} \right)\,\sqrt{\mbf E_s} \\
&\lesssim\,\norm{ \nabla \left( \frac{\bt}{\sqrt \o}\right)}_{L^\infty}\,\norm{\frac{\bt}{\sqrt \o}\,\nabla f }_{H^s}\,\sqrt{\mbf E_s} \\
&\qquad\qquad\qquad +\,
(1+t)^{[s]+\frac{3}{2}}\,\left(1+\norm{\bt}_{L^\infty}\right)\,\left(1+\norm{\nabla\o}_{L^\infty}^{[s]}\right)
\norm{ \nabla \left(\frac{\bt}{\sqrt \o}\,\nabla f\right) }_{L^\infty }\,\mbf E_s\,.
\end{align*}
At this point, Proposition \ref{p:key} enters into play in a fundamental way. Indeed, a direct application of this result allows us to bound
\begin{align}
\label{est:from-key-Prop}
\norm{\frac{\bt}{\sqrt \o}\,\nabla f }_{H^s}\,&\lesssim\,\left(\sum_{j\in\Z}2^{2js}\norm{\Delta_j\left(\frac{\bt}{\sqrt \o}\,\nabla f\right) }_{L^2}^2\right)^{1/2} \\
\nonumber
&\qquad\qquad
+\norm{\nabla f}_{L^\infty}\norm{\frac{\bt}{\sqrt{\o}}}_{H^s}+\norm{\nabla\left(\frac{\bt}{\sqrt{\o}}\right)}_{L^\infty}\norm{f}_{H^s} \\
\nonumber
&\lesssim\,\sqrt{F_s}\,+\,\sqrt{\mbf E_s}\,(1+t)^{[s]+\frac{3}{2}}\,
\left(1+\norm{\bt}_{L^\infty}\right)\,\left(1+\norm{\big(\nabla u,\nabla\o,\nabla\bt\big)}_{L^\infty}^{[s]}\right)\,.
\end{align}
This estimate in turn implies that
\begin{align}
\label{est:T_2}
\left|\mc T_2\right|\,&\lesssim\,(1+t)^{\frac{3}{2}}\,\left(1+\norm{\bt}_{L^\infty}\right)\,\norm{\big(\nabla\o,\nabla\bt\big)}_{L^\infty}\,\sqrt{F_s}\,\sqrt{\mbf E_s} \\
\nonumber
&\qquad +\,(1+t)^{[s]+3}\,\left(1+\norm{\bt}^2_{L^\infty}\right)\,\left(1+\norm{\big(\nabla u,\nabla\o,\nabla\bt\big)}_{L^\infty}^{[s]+1}\right)\,\mbf E_s \\
\nonumber
&\qquad\qquad +\,(1+t)^{[s]+\frac{3}{2}}\,\left(1+\norm{\bt}_{L^\infty}\right)\,\left(1+\norm{\nabla\o}_{L^\infty}^{[s]}\right)
\norm{ \nabla \left(\frac{\bt}{\sqrt \o}\,\nabla f\right) }_{L^\infty }\,\mbf E_s\,.
\end{align}

The other term appearing in \eqref{eq:C^2_dec-2}, instead, needs a further decomposition.
As a matter of fact, we notice that we can write it
\[
\sum_{j \in \Z}2^{2js}\int_{\Omega}\comm{ \Delta_j}{\nabla \left(\frac{\beta}{\sqrt \o} \right) }  \left(  \frac{\bt}{\sqrt{\o}} \nabla f \right)
\; \Delta_jf\,  \dx\,=\,\mc T_3\,+\,\mc T_4\,,
\]
where we have defined
\begin{align*}
\mc T_3\,&:=\,\sum_{j \in \Z}2^{2js}\int_{\Omega}\comm{ \Delta_j}{\frac{\bt}{\sqrt{\o}} \,\nabla \left(\frac{\beta}{\sqrt \o} \right) }\cdot\nabla f
\; \Delta_jf\,  \dx\,, \\
\mc T_4\,&:=\,\sum_{j \in \Z}2^{2js}\int_{\Omega}\nabla\left(\frac{\beta}{\sqrt \o}\right)\cdot\left(\comm{\frac{\bt}{\sqrt{\o}} }{ \Delta_j}\,\nabla f\right)
\; \Delta_jf\,  \dx\,.
\end{align*}
The advantage is that we can now estimate $\mc T_3$ and $\mc T_4$ as above, by use of the Cauchy-Schwarz inequality and of Lemma \ref{l:comm}.
Therefore we get
\begin{align*}
\left|\mc T_3\right|\,&\lesssim\,\left(\norm{\nabla  \left( \frac{\bt}{\sqrt{\o}}\, \nabla \left(  \frac{\bt}{\sqrt{\o}} \right)  \right)  }_{L^\infty } \norm{f}_{H^s} +
\norm{  \frac{\bt}{\sqrt{\o}}\, \nabla \left(  \frac{\bt}{\sqrt{\o}} \right)    }_{H^s }\, \norm{\nabla f}_{L^\infty}\right)\,\norm{f}_{\dot H^s} \\
\left|\mc T_4\right|\,&\lesssim\,\norm{\nabla \left(\frac{\bt}{\sqrt \o}\right)}_{L^\infty}\, \norm{f}_{\dot H^s}\,
\left( \norm{\nabla\left(\frac{\bt}{\sqrt \o}\right)}_{L^\infty}\, \norm{f}_{H^s} + \norm{ \frac{\bt}{\sqrt \o}}_{H^s}\, \norm{\nabla f }_{L^\infty }  \right)\,.
\end{align*}
Now, it is a long but fairly straightforward computation to see that
\begin{align*}
\norm{\nabla  \left( \frac{\bt}{\sqrt{\o}}\, \nabla \left(  \frac{\bt}{\sqrt{\o}} \right)  \right)  }_{L^\infty } \,&\lesssim\,
(1+t)^{\frac{3}{2}}\,\left(1+\norm{\bt}_{L^\infty}\right)\,
\left(\norm{\nabla\left(\frac{\bt}{\sqrt{\o}}\, \nabla\bt\right)}_{L^\infty}\,+\,\norm{\nabla\left(\frac{\bt}{\sqrt{\o}}\, \nabla\o\right)}_{L^\infty}\right) \\
&\qquad +\,(1+t)^3\,\left(1+\norm{\bt}_{L^\infty}^2\right)\,\norm{\big(\nabla\o,\nabla\bt\big)}_{L^\infty}^2\,.
\end{align*}
Similarly, we can bound
\begin{align*}
\norm{  \frac{\bt}{\sqrt{\o}}\, \nabla \left(  \frac{\bt}{\sqrt{\o}} \right)    }_{H^s }\,&\lesssim\,
(1+t)^{\frac{3}{2}}\,\left(1+\norm{\bt}_{L^\infty}\right)\sum_{g\in \{\o,\bt\}}\norm{\frac{\bt}{\sqrt{\o}}\,\nabla g}_{H^s} \\
&\qquad +\,(1+t)^2\,\norm{\bt}_{L^\infty}\,\norm{\big(\nabla\o,\nabla\bt\big)}_{L^\infty}\,\norm{\left(\bt,\frac{1}{\sqrt{\o}}\right)}_{H^s} \\
&\qquad\qquad+\,
(1+t)^{\frac{1}{2}}\,\norm{\bt}^2_{L^\infty}\,\norm{\nabla\o}_{L^\infty}\,\norm{\frac{1}{\o^{3/2}}}_{H^s}\,,
\end{align*}
which yields, by use of \eqref{est:1/o}, of inequality \eqref{est:from-key-Prop} and of Lemma \ref{l:prod}, the estimate
\begin{align*}
\norm{  \frac{\bt}{\sqrt{\o}}\, \nabla \left(  \frac{\bt}{\sqrt{\o}} \right)    }_{H^s }\,&\lesssim\,
(1+t)^{\frac{3}{2}}\,\left(1+\norm{\bt}_{L^\infty}\right)\,\sqrt{F_s} \\
&\qquad +\,(1+t)^{[s]+4}\,\left(1+\norm{\bt}^2_{L^\infty}\right)\,\left(1+\norm{\big(\nabla\o,\nabla\bt\big)}_{L^\infty}^{[s]+1}\right)\,\sqrt{\mbf E_s}\,.
\end{align*}
Putting all these inequalities together, we infer that
\begin{align}
\label{est:T_3}
\left|\mc T_3\right|\,&\lesssim\,(1+t)^{\frac{3}{2}}\,\left(1+\norm{\bt}_{L^\infty}\right)\,\left(\sum_{g\in\{\o,\bt\}}
\norm{\nabla\left(\frac{\bt}{\sqrt{\o}}\, \nabla g\right)}_{L^\infty}\right)\,\mbf E_s \\
\nonumber
&\qquad + \,(1+t)^{[s]+4}\,\left(1+\norm{\bt}^2_{L^\infty}\right)\,\left(1+\norm{\big(\nabla\o,\nabla\bt\big)}_{L^\infty}^{[s]+2}\right)\,\mbf E_s \\
\nonumber
&\qquad\qquad +\,
(1+t)^{\frac{3}{2}}\,\left(1+\norm{\bt}_{L^\infty}\right)\,\left(1+\norm{\big(\nabla u,\nabla\o,\nabla\bt\big)}_{L^\infty}\right)\,\sqrt{F_s}\,\sqrt{\mbf E_s}\,.
\end{align}
As for $\mc T_4$, by direct computations and using inequality \eqref{est:b/o} again, we obtain
\begin{align}
\label{est:T_4}
\left|\mc T_4\right|\,&\lesssim\,\mbf E_s\,(1+t)^3\,\left(1+\norm{\bt}_{L^\infty}^2\right)\,\norm{\big(\nabla\o,\nabla\bt\big)}_{L^\infty}^2 \\
\nonumber
&\qquad
+\,\mbf E_s\,(1+t)^{[s]+3}\,\left(1+\norm{\bt}_{L^\infty}^2\right)\,\left(1+\norm{\big(\nabla u,\nabla\o,\nabla\bt\big)}_{L^\infty}^{[s]+2}\right) \\
\nonumber
&\lesssim\,\mbf E_s\,(1+t)^{[s]+3}\,\left(1+\norm{\bt}_{L^\infty}^2\right)\,\left(1+\norm{\big(\nabla u,\nabla\o,\nabla\bt\big)}_{L^\infty}^{[s]+2}\right)\,.
\end{align}

\medbreak
All in all, we have written
\[
 \mds C^2[f,f]\,=\,\sum_{j \in \Z} 2^{2js}\int_{\Omega}\mds C^2_{f,j}\, \Delta_jf \,\dx\,=\,\sum_{j=1}^4\mc T_j\,,
\]
where the terms $\mc T_1,\ldots \mc T_4$ satisfy the bounds \eqref{est:T_1}, \eqref{est:T_2}, \eqref{est:T_3} and \eqref{est:T_4}. In particular, this implies that,
after an application of the Young inequality, for any $\de>0$ to be fixed later and for any $f\in\{u,\o,\bt\}$, we have
\begin{align}
\label{est:C^2}
\sum_{f\in\{u,\o,\bt\}}\left|\mds C^2[f,f]\right|\,&\leq\,3\,\de\,F_s\,+\,C\,(1+t)^{[s]+4}\,\left(1+\norm{\bt}^2_{L^\infty}\right)\,
\left(1+\norm{\big(\nabla u,\nabla\o,\nabla\bt\big)}_{L^\infty}^{[s]+2}\right)\,\mbf E_s \\
\nonumber
&\qquad +\,C\,(1+t)^{[s]+\frac{3}{2}}\,\left(1+\norm{\bt}_{L^\infty}\right)\,\left(1+\norm{\nabla\o}_{L^\infty}^{[s]}\right) \\
\nonumber
&\qquad\qquad\qquad\qquad\qquad
\times\,\left(\sum_{G\in\{\D u,\nabla\o,\nabla\bt\}}\norm{ \nabla \left(\frac{\bt}{\sqrt \o}\,G\right) }_{L^\infty }\right)\,\mbf E_s\,,
\end{align}
where the multiplicative constant $C>0$ only depends on $\de$ and on the various parameters of the problem, but not on the solution.

\subsection{End of the argument} \label{ss:closing-est}

At this point, to conclude our argument and close the estimates in some (possibly small) time interval $[0,T]$, we have to insert inequalities
\eqref{est:C^1}, \eqref{est:C^3_fin}, \eqref{est:C^4_fin} and \eqref{est:C^2} into estimate \eqref{est:tot-E_partial}.

First of all, using again the Young inequality at the right place, we see that
\begin{align*}
&\left|\int^t_0\Big(\sum_{f\in\{u,\o,\bt\}}\left(\mds C^1[f,f]+\mds C^2[f,f]\right)+\mds C^3[u,\bt]+\mds C^4[\bt,\bt]\Big)\,\dd\t\right| \\
&\;\leq\,5\,\de\,\int^t_0F_s\,\dd\t\,+\,C\int^t_0(1+\t)^{[s]+4}\,\left(1+\norm{\bt}^2_{L^\infty}\right)\,
\left(1+\norm{\big(\nabla u,\nabla\o,\nabla\bt\big)}_{L^\infty}^{[s]+2}\right)\,\mbf E_s\,\dd\t \\
&\quad +\,C\int^t_0(1+\t)^{[s]+\frac{3}{2}}\,\left(1+\norm{\bt}_{L^\infty}\right)\,\left(1+\norm{\nabla\o}_{L^\infty}^{[s]}\right)
\left(\sum_{G\in\{\D u,\nabla\o,\nabla\bt\}}\norm{ \nabla \left(\frac{\bt}{\sqrt \o}\,G\right) }_{L^\infty }\right)\,\mbf E_s\,\dd\t\,,
\end{align*}
for a new ``universal'' constant $C>0$.

In view of this bound, using \eqref{eq:obound} for controlling $\o$ in $L^\infty$ and taking $\de>0$ small enough, from inequality \eqref{est:tot-E_partial} we get
\begin{align}
\label{est:tot-E_part2}
\mbf E_s(t)\,+\,\int^t_0F_s(\t)\,\dd\t\,&\leq\,C_1\,\mbf E_s(0)\,+\,C_2\int^t_0\,\Big(\xi(\t)+\Lambda(\t)\Big)\,\mbf E_s(\t)\,\dd\t\,,
\end{align}
where, for $t\geq0$, we have defined the functions
\begin{align*}
 \xi(t)\,&:=\,(1+t)^{[s]+4}\,\left(1+\norm{\bt}^2_{L^\infty}\right)\,\left(1+\norm{\big(\nabla u,\nabla\o,\nabla\bt\big)}_{L^\infty}^{[s]+2}\right)\,, \\
\Lambda(t)\,&=\,(1+t)^{[s]+\frac{3}{2}}\,\left(1+\norm{\bt}_{L^\infty}\right)\,\left(1+\norm{\nabla\o}_{L^\infty}^{[s]}\right)
\left(\sum_{G\in\{\D u,\nabla\o,\nabla\bt\}}\norm{ \nabla \left(\frac{\bt}{\sqrt \o}\,G\right) }_{L^\infty }\right)\,.
\end{align*}
Recall that the constants $C_1>0$ and $C_2>0$ only depend on $\big(d,s,\nu,\alpha_1,\ldots \alpha_4,\o_*,\o^*)$, but not on the solution.

Thus, an application of the Gr\"onwall inequality yields
\[
\forall\, t\geq0\,,\qquad \mbf E_s(t)\,\leq\,C_1\,\mbf E_s(0)\,\exp\left(C_2\int^t_0\Big(\xi(\t)\,+\,\Lambda(\t)\Big)\,\dd\t\right)\,.
\]
Coming back to \eqref{est:tot-E_part2} and using the previous bound, we discover that
\begin{equation}
\label{est:tot-E}
\forall\,t\geq0\,,\qquad \mbf E_s(t)\,+\,\int^t_0F_s(\t)\,\dd\t\,\leq\,C_1\,\mbf E_s(0)\,
\left(1\,+\,\exp\left(C_2\int^t_0\Big(\xi(\t)\,+\,\Lambda(\t)\Big)\,\dd\t\right)\right)\,.
\end{equation}

Concluding the argument from this inequality is now a standard matter.
Let us define the time $T>0$ as
\begin{equation*}
T\,:=\,\sup\left\{t>0\;\Big|\qquad C_2\int^t_0\Big(\xi(\t)\,+\,\Lambda(\t)\Big)\,\dd\t\,\leq\,\log 2\right\}\,.
\end{equation*}
Then, from \eqref{est:tot-E} we deduce that
\begin{equation}
\label{est:tot-E_T}
\forall\,t\in[0,T]\,,\qquad \mbf E_s(t)\,+\,\int^t_0F_s(\t)\,\dd\t\,\leq\,3\,C_1\,\mbf E_s(0)\,.
\end{equation}
Now, on the one hand, we want to deduce a suitable lower bound for $T$ and, on the other hand, we want to establish a continuation criterion
for solutions to system \eqref{eq:kolm_d}.

\paragraph{Lower bound for the lifespan \\}

By definition of the functions $\xi(t)$ and $\Lambda(t)$ and Sobolev embeddings, it is easy to see that, for any time $t\geq0$, one has
\begin{align*}
\xi(t)\,&\lesssim\,(1+t)^{[s]+4}\,\left(1+\mbf E_s(t)\right)^{[s]+4} \\ 
\Lambda(t)\,&\lesssim\,(1+t)^{[s]+\frac{3}{2}}\,\left(1+\norm{\bt}_{L^\infty}\right)\,\left(1+\norm{\nabla\o}_{L^\infty}^{[s]}\right)
\left(\sum_{G\in\{\D u,\nabla\o,\nabla\bt\}}\norm{\left(\frac{\bt}{\sqrt \o}\,G\right) }_{H^s}\right) \\
&\lesssim(1+t)^{[s]+\frac{3}{2}}\,\left(1+\norm{\bt}_{L^\infty}\right)\,\left(1+\norm{\nabla\o}_{L^\infty}^{[s]}\right)\,\sqrt{F_s(t)} \\
&\qquad\qquad\qquad\qquad\qquad\qquad \,+\,
(1+t)^{2[s]+3}\,\left(1+\norm{\bt}^2_{L^\infty}\right)\,\left(1+\norm{\nabla\o}_{L^\infty}^{2[s]}\right)\,\sqrt{\mbf E_s}  \\
&\leq\,\de\,F_s(t)\,+\,\dfrac{C}{\de}
(1+t)^{2[s]+3}\,\big(1+\mbf E_s(t)\big)^{2[s]+3}\,,
\end{align*}
where we have also used inequality \eqref{est:from-key-Prop} for controlling $\Lambda(t)$. Notice that the previous inequality holds true for any fixed $\de>0$
and that the constant $C>0$ depends all the parameters of the problem.
Therefore, from estimate \eqref{est:tot-E_T} we deduce that, for any $t\in[0,T]$, one must have
\[
\int^t_0\Big(\xi(\t)\,+\,\Lambda(\t)\Big)\,\dd\t\,\leq\,\frac{K_1}{\de}\,t\,(1+t)^{2[s]+3}\,\big(1+\mbf E_s(0)\big)^{2[s]+3}\,+\,\de\,K_2\,\mbf E_s(0)\,,
\]
for suitable positive constants $K_1$ and $K_2$, only depending on the parameters of the problem.
In particular, by definition of $T$, at time $t=T$ it must hold
\begin{equation} \label{est:T-life}
\frac{K_1}{\de}\,T\,(1+T)^{2[s]+3}\,\big(1+\mbf E_s(0)\big)^{2[s]+3}\,+\,\de\,K_2\,\mbf E_s(0)\,\geq\,K_3\,,
\end{equation}
where $K_3\,=\,\log 2/C_2$ only depends on $\big(d,s,\nu,\alpha_1,\ldots \alpha_4,\o_*,\o^*\big)$. 
At this point, we choose
\[
\de\,=\,\frac{K_3}{2\,K_2\,\mbf E_s(0)}\,.
\]
Then, from \eqref{est:T-life} we obtain that
\begin{equation} \label{est:ineq-for-T}
\frac{2}{K_3}\,K_1\,K_2\,\mbf E_s(0)\,T\,(1+T)^{2[s]+3}\,\big(1+\mbf E_s(0)\big)^{2[s]+3}\,\geq\,\frac{K_3}{2}\,.
\end{equation}

At this point, if $T\geq 1$ we are done. So, let us assume that $T\leq1$. Then the previous estimate in particular implies that
\[
\mbf E_s(0)\,T\,\big(1+\mbf E_s(0)\big)^{2[s]+3}\,\geq\,K_0\qquad\qquad \Longrightarrow\qquad\qquad 
T\,\geq\,\frac{K_0}{\mbf E_s(0)\,\big(1+\mbf E_s(0)\big)^{2[s]+3}}\,,
\]
for a suitable ``universal'' constant $K_0$. This proves the lower bound on the lifespan of the solution claimed in Theorem \ref{t:cont}.

\paragraph{Continuation criterion \\}
Let us now turn our attention to the proof of the continuation criterion. This part will conclude the proof of Theorem \ref{t:cont}.

To begin with, we remark that, as a direct consequence of inequality \eqref{est:tot-E}, we have the following claim: given a time $0<T<+\infty$,
if one has
\begin{equation} \label{est:cont_1}
\int^T_0\Big(\xi(t)\,+\,\Lambda(t)\Big)\,\dt\,<\,+\infty\,,
\end{equation}
then the solution remains bounded in $H^s$ on the time interval $[0,T]$. Hence, by classical arguments, this solution may be continued beyond the time $T$
into a solution possessing the same regularity.

The previous argument already provides us with a first continuation criterion. Our goal now is to refine it, in order to match the one claimed in Theorem \ref{t:cont}.
As a matter of fact, as $T<+\infty$ has finite value by assumption, it is easy to remove the factors $(1+t)$ from that criterion and
see that condition \eqref{est:cont_1} holds true if and only if
\[
 \int^T_0\Big(\wtilde \xi(t)\,+\,\wtilde\Lambda(t)\Big)\,\dt\,<\,+\infty\,,
\]
where this time we have defined
\begin{align*}
\wtilde \xi(t)\,&:=\,\left(1+\norm{\bt}^2_{L^\infty}\right)\,\left(1+\norm{\big(\nabla u,\nabla\o,\nabla\bt\big)}_{L^\infty}^{[s]+2}\right)\,, \\
\wtilde\Lambda(t)\,&=\,\left(1+\norm{\bt}_{L^\infty}\right)\,\left(1+\norm{\nabla\o}_{L^\infty}^{[s]}\right)
\left(\sum_{G\in\{\D u,\nabla\o,\nabla\bt\}}\norm{ \nabla \left(\frac{\bt}{\sqrt \o}\,G\right) }_{L^\infty }\right)\,.
\end{align*}

In addition, using the Poincar\'e-Wirtinger type inequality \eqref{est:interp_2} and estimate \eqref{est:bl2}, we see that we can bound
\begin{align*}
\int_0^{T} \norm{\bt}_{L^\infty}^2\,\dt\, &\lesssim\,\int_0^{T}\left( \oline{\bt}^2\,+\,
\left\| \bt \right\|_{L^2}^{4/(d+2)}\;\left\| \nabla \bt \right\|_{L^\infty}^{2d/(d+2)}  \right)\, \dd\t \\
&\lesssim\, \norm{\bt_0}_{L^2}^2\, T\, +\,C\left(\norm{\big(u_0,\bt_0\big)}_{L^2}\right)\,\int_0^{T} \norm{\nabla \bt }_{L^\infty}^{2d/(d+2)}\,\dd\t\,. 
\end{align*}
Therefore, as $2d/(d+2)\,<\,2$ on the one hand we can bound
\[
\int^T_0\wtilde\xi(t)\,\dt\,\lesssim\,\int^T_0\left(1+\norm{\big(\nabla u,\nabla\o,\nabla\bt\big)}_{L^\infty}^{[s]+4}\right)\,\dt
\]
and, on the other hand, we also have
\[
 \int^T_0\wtilde\Lambda(t)\,\dt\,\lesssim\,\int^T_0\left(1+\norm{\nabla\bt}_{L^\infty}\right)\,\left(1+\norm{\nabla\o}_{L^\infty}^{[s]}\right)
\left(\sum_{G\in\{\D u,\nabla\o,\nabla\bt\}}\norm{ \nabla \left(\frac{\bt}{\sqrt \o}\,G\right) }_{L^\infty }\right)\,\dt\,.
\]
From this last bounds, the continuation criterion of Theorem \ref{t:cont} easily follows.

\subsection{Analysis of the pressure} \label{ss:pressure}

Notice that, in our derivation of \tsl{a priori} estimates, the pressure term does not play any role. As a matter of fact, this term simply disappears
in our energy method, because of the orthogonality with $u$, owing to the divergence-free constraint $\div u=0$.

Nonetheless, the pressure gradient $\nabla\pi$ appearing in equations \eqref{eq:kolm_d} as an unknown of the problem, for the sake
of completeness we want to establish its regularity. This is the goal of this subsection.

More precisely, we want to prove that
\[
\nabla\pi\,\in\,L^\infty\big([0,T];H^{s-1}(\Omega)\big)\,,
\]
where $T>0$ is the time defined above, so that inequality \eqref{est:tot-E_T} holds true.

\medbreak
First of all, let us derive an equation for $\nabla\pi$. By taking the divergence of the first equation in \eqref{eq:kolm_d}, we see that the pressure
satisfies the following elliptic problem: 
\begin{equation}\label{elliptic}
	-\Delta \pi\, =\,-\,\nu\,\div\left(  \div \left( \frac{k}{\o} \D u  \right)  - u \cdot \nabla u\right)\,.
\end{equation}
Because the average $\oline{\nabla\pi}=0$ vanishes thanks to periodic boundary conditions, the previous equation implies that
\begin{equation} \label{eq:formula-pi}
\nabla\pi\,=\,-\nu\,\nabla(-\Delta)^{-1}\div\div \left( \frac{k}{\o} \D u  \right)\,+\,\nabla(-\Delta)^{-1}\div\big(u\cdot\nabla u\big)\,.
\end{equation}

Observe that, because of the fact that $\div u=0$, we have $\div\big(u\cdot\nabla u\big)\,=\,\nabla u:\nabla u$. Taking advantage of this cancellation,
one can establish that
\[
\norm{\nabla(-\Delta)^{-1}\div\big(u\cdot\nabla u\big)}_{H^s}\,\lesssim\,\norm{\nabla u:\nabla u}_{H^{s-1}}\,\lesssim\,\norm{u}_{H^s}^2\,,
\]
which is finite uniformly for $t\in[0,T]$.

As for the other term appearing in identity \eqref{eq:formula-pi}, we take advantage of another fundamental cancellation, still deriving from
the divergence-free constraint $\div u=0$. For notational convenience, let us set
\[
\frac{k}{\o}\,=\,\alpha^2\,,\qquad\qquad \mbox{ with }\qquad \alpha\,\in\,L^\infty\big([0,T];H^s(\Omega)\big)\,.
\]
Hence, let us compute
\begin{align*}
\div \div\left(\alpha^2\,\D u\right)\,&=\,\sum_{j,k}\d_j\,\d_k\Big(\alpha^2\,\big(\d_ku_j\,+\,\d_ju_k\big)\Big)\,=\,
\sum_k\d_k\Big(\d_ku\cdot\nabla\alpha^2\Big)\,+\,\sum_j\d_j\Big(\d_ju\cdot\nabla\alpha^2\Big)\,.
\end{align*}
Observe that, as $H^{s-1}(\Omega)$ is a Banach algebra, we have that the product $\d_ku\cdot\nabla\alpha^2$ belongs to $L^\infty\big([0,T];H^{s-1}\big)$. This implies that
\[
\nabla(-\Delta)^{-1}\div\div\left( \frac{k}{\o} \D u  \right)\,\in\,L^\infty\big([0,T];H^{s-1}\big)\,,
\]
thus completing the proof of our claim.

\section{Local existence and uniqueness of solutions} \label{s:proof}

In this section, we perform the proof to Theorem \ref{t:d_wp}. More precisely, starting from the \tsl{a priori} bounds of Section \ref{s:a-priori},
we rigorously derive existence and uniqueness of local in time solutions to system \eqref{eq:kolm_d}, related to given initial data $\big(u_0,\o_0,k_0\big)$.
The existence issue is dealt with in Subsection \ref{ss:exist}, while uniqueness is proved in Subsection \ref{ss:uniqueness}.

Our argument follows the strategy adopted in \cite{F-GB} for treating the $1$-D model. Although most arguments can be reproduced similarly in our context,
we present the proofs here for the sake of completeness. However, as in \cite{F-GB} the proofs are discussed in detail, when appropriate,
we will omit to give the full details and rather refer the reader to that paper.

\subsection{Local existence of solutions} \label{ss:exist}

The proof of existence of a solution is carried out in three main steps. First of all, we construct approximated solutions to system \eqref{eq:kolm_d},
by avoiding the appearing of the ``vacuum region'' $\big\{k_0(x)=0\big\}$.
Then, we show uniform bounds for the family of solutions constructed by approximation, thus convergence (up to extraction) to some target state $\big(u,\o,k\big)$.
Finally, by a compactness argument we prove that the target state $\big(u,\o,k\big)$ is indeed a solution of the original system \eqref{eq:kolm_d}.
Observe that, in the convergence step, we can avoid the appearing of the pressure function $\nabla\pi$, as the weak formulation
of the momentum equation uses divergence-free test functions. A discussion about the regularity of $\nabla\pi$ will arise only at the end of the argument.

Before starting the proof, some notation is in order.
We will deal with sequences of solutions, denoted $(f_\eps)_{\eps}$. Given some normed space $X$, we simply write $(f_\eps)_{\eps} \subset X$ to mean
that the sequence is also \emph{bounded} in that space. If the sequence is not bounded, we will adopt the different notation $\forall\,\veps>0\,,\ f_\veps\in X$.
For simplicity, we will sometimes use the notation $ L^p_T(X) := L^p\big([0,T]; X\big)$.
 
\medbreak
This having been said, let us start the proof of the existence of a solution, given some initial state $\big(u_0,\o_0,k_0\big)$ verifying the assumptions
stated in Theorem \ref{t:d_wp}.

We begin by removing the degeneration created by the possible vanishing of turbulent kinetic energy $k_0$. For this, we lift the initial data: for $0 <\eps < 1$, we define
$$ k_{0,\eps}\, :=\, \left(\sqrt{k_0} + \eps\right)^2\,.$$
From the initial regularity $\sqrt{k_0} \in H^s(\Omega)$, it is easy to see that 
$$ \big(k_{0,\eps}\big)_{\eps} \,\subset\, H^s(\Omega)\,, \qquad\qquad k_{0,\eps}\,\geq\,\veps^2\,>\,0\,. $$
Thus, for any fixed $0<\eps<1$, we can solve the original system \eqref{eq:kolm_d} with respect to the initial datum
$$\big(u_0,\o_0,k_{0,\eps}\big)\, \in\, H^s(\Omega) \times  H^s(\Omega) \times H^s(\Omega)\,,$$
for instance by using Theorem 1 in \cite{Kos} (see also \cite{Kos-Kub}).

Observe that the  solution $\big(u_{\eps},\o_{\eps},k_{\eps}\big)$ in \cite{Kos}-\cite{Kos-Kub}
is constructed \tsl{via} a Galerkin method, hence the approximation of those solutions are smooth
at any step of the Galerkin construction. In particular, for those approximations the computations performed in Section \ref{s:a-priori} are fully justified; we deduce
that the \tsl{a priori} estimates we have established therein are satisfied by the Galerkin approximations, hence inherited also
by the ``true'' solution $\big(u_\veps,\o_\veps,k_\veps\big)$.

The previous argument ensures us that some uniform-in-$\veps$ properties hold true: let us review them.
Firstly, thanks to the lower bound \eqref{est:lower-T} for the lifespan of the solutions, one has that
$$ T \,:=\, \inf_{\eps \in\,]0,1]} T_\eps \,>\,0\,.$$
In particular, all the solutions $\big(u_\veps,\o_\veps,k_\veps\big)_\veps$ are defined on a common time interval $[0,T]$.

Additionally, the pointwise bounds described in \eqref{eq:obound} imply that
$$ \big(\o_\eps\big)_{\eps} \, \subset\, L^\infty\big([0,T] \times \Omega\big)\,, \qquad\qquad \mbox{ with }\qquad 0\, <\, \o_\eps(t,x)\, \leq\, \o^\ast\,,$$
whereas inequality \eqref{eq:kbound} yields
$$  k_\eps(t,x)\, >\,0\,, $$
where the previous esimates hold for all $(t,x) \in [0,T] \times \Omega$.

Furthermore, the uniform estimate \eqref{est:tot-E_T} gives us
$$ \left(u_\eps, \o_\eps, \sqrt{k_\eps}\right)_\eps\, \subset\, L_T^\infty(H^s) \times L_T^\infty(H^s)\times L_T^\infty(H^s)\,.$$
Then, it is direct in our setting to deduce that $\big(k_\eps\big)_\eps \,\subset\, L_T^\infty(H^s) $ as well. From the uniform estimate \eqref{est:tot-E_T}
and inequality \eqref{est:from-key-Prop}, we also obtain that
\begin{equation}\label{keynorms}
\left(\sqrt{\dfrac{k_\eps}{\o_\eps}}\, \D u_\eps\right)_\eps\,, \quad \left(\sqrt{\dfrac{k_\eps}{\o_\eps}}\, \nabla \o_\eps\right)_\eps\,, 
\quad \left(\sqrt{\dfrac{k_\eps}{\o_\eps}}\, \nabla \sqrt{k_\eps}\right)_\eps \quad \subset \; L^2_T(H^s)\,.
\end{equation}

Thus, by Banach-Alaoglu theorem we deduce the existence of a triplet $\big(u,\o,k\big)\in L^\infty_T(H^s)\times L^\infty_T(H^s)\times L^\infty_T(H^s)$
such that, up to a suitable extraction of a subsequence, $\big(u_\veps,\o_\veps,k_\veps\big)$ converges to $\big(u,\o,k\big)$ in the weak-$*$ topology
of that space.
Our next goal is to prove prove that $\big(u,\o,k\big)$ is indeed the sought solution to the original system \eqref{eq:kolm_d}.

In order to reach our goal, we are going to use a compactness argument. Let us start by considering the velocity fields $u_\veps$.
Recall the equation for $u_\eps$:
\begin{equation*}
	 \d_tu_\eps \,+\,(u_\eps\cdot\nabla) u_\eps\,+\,\nabla\pi_\eps\,-\,\nu\,\div\left(\dfrac{k_\eps}{\o_\eps}\,\D u_\eps\right)\,=\,0\, .
\end{equation*}
Repeating the analysis performed in Subsection \ref{ss:pressure}, we infer that
$$\big( \nabla \pi_\veps\big)\,\subset\,L^\infty_T(H^{s-1})\,.$$
From this property, we easily find that
$$\big( \d_t u_\eps\big)_\eps \,\subset\, L^\infty_T(H^{s-1})
\qquad\qquad \Longrightarrow\qquad\qquad 
\big(u_\eps \big)_{\eps}\, \subset\, L_T^\infty (H^s) \cap W_T^{1,\infty}(H^{s-1})\,.$$
Hence, Ascoli-Arzel\`a theorem implies the compact inclusion $\big(u_\eps\big)_\eps\,\subset \subset\,C_T(H^{s-1})$. By interpolation, we immediately deduce the strong convergence
$$
u_\eps\, \to\, u \qquad\qquad \text{ in }\qquad  C_T(H^\s)\,, \quad \text{ for any }\ \ 0 \leq \s < s\,.$$
If follows from the previous regularities and Sobolev embeddings that, then, $u_\eps$ converges pointwise together with its first order derivatives:
\begin{equation}\label{pointwiselimit}
u_\eps\, \to\, u \qquad \text{ and } \qquad D u_\eps\, \to\, D u \qquad \qquad \text{ everywhere in } \ [0,T] \times \Omega\,.
\end{equation}

It goes without saying that an analogous analysis can be made for the sequences $\big(\o_\eps\big)_\eps$ and $\big(k_\eps\big)_\eps$, yielding
pointwise convergence for them and their first order derivatives. At that point, it is an easy matter to pass to the limit in the weak formulation of the equations
and see that $\big(u,\o,k\big)$ is indeed a solution to the original system \eqref{eq:kolm_d}.
With this at hand, we can conclude also that the gradient of the pressure $\nabla\pi$, recovered from the target velocity
and its equation, has the claimed regularity, namely
$\nabla \pi \in L^\infty_T(H^{s-1})\cap C_T(H^{\s-1})$, for any $\s<1$.

Observe that, thanks to the previous analysis and following the arguments in \cite{F-GB},
it is not difficult to conclude also that the triplet $\big(u,\o,\sqrt{k}\big)$ solves the reformulated
system \eqref{eq:kolm_dnew} and possesses the claimed regularity properties.

\subsection{Uniqueness of solutions} \label{ss:uniqueness}

Now, we focus on the proof of uniqueness of solutions in the claimed functional framework, namely in the space
\begin{align*}
\mathbb{X}_T(\Omega)\,:=\,\Big\{ (u,\nabla\pi,\o,k) \; \Big| & \quad u\,,\, \o\,,\, \sqrt{k} \,\in\, C\big([0,T];L^2(\Omega)\big)\,, \quad
\nabla\pi\,\in\,L^\infty\big([0,T];H^{-1}(\Omega)\big)\,, \\
&\qquad\quad \o\,,\, \o^{-1}\,,\, k \,\in\, L^\infty\big([0,T]\times \Omega\big)\,, \quad \o >0\,, \quad k \ge 0\,, \\
&\qquad\qquad\qquad \div u\,=\,0\,, \qquad \nabla u\,,\, \nabla \o\,,\, \nabla\sqrt{k}\, \in\, L^\infty\big([0,T]\times \Omega\big)
\Big\}\,.
\end{align*}

First of all, we notice that the original system \eqref{eq:kolm_d} and the modified system \eqref{eq:kolm_dnew} are equivalent
in the space $\mbb X_T(\Omega)$. This claim is justified thanks to Lemma 4.1 of \cite{F-GB}, whose proof can be adapted with no difficulties
to the multi-dimensional case.

With this property at hand, we deduce that, for establishing uniqueness of solutions for system \eqref{eq:kolm_d}, it is enough to prove it
for solutions $\big(u,\o,\bt\big)$ to system \eqref{eq:kolm_dnew}.
Thus, the uniqueness statement of Theorem \ref{t:d_wp} is a consequence of the following result.

\begin{thm}
Consider two triplets $\big(u_1,\o_1,\bt_1\big)$ and $\big(u_2,\o_2,\bt_2\big)$ and assume that they are both solutions to \eqref{eq:kolm_dnew},
related to some initial datum $\big(u_{0,j},\o_{0,j},\bt_{0,j}\big)$, for $j=1,2$. Assume also that, for some time $T>0$ and $j=1,2$, it holds that 
	\begin{align*}
\big(u_j,\o_j,\bt_j\big)\, \in\,\wtilde{\mbb X}_T\,:=\,
\Big\{ (u,\o,\bt) \; \Big| & \quad \o\,,\, \o^{-1}\,,\, \bt \,\in\, L^\infty\big([0,T]\times \Omega\big)\,, \quad \o >0\,, \quad \bt \ge 0\,,  \\
&\qquad\qquad
u\,\in\,L^\infty\big([0,T];L^2(\Omega)\big)\,, \quad \div u\,=\,0\,, \\
&\qquad\qquad\qquad\qquad\qquad \nabla u\,,\, \nabla \o\,,\, \nabla\bt\, \in\, L^\infty\big([0,T]\times \Omega\big)
\Big\}\,.
	\end{align*}
Define the difference of the solutions as 
\begin{equation*}
	U\,:=\, u_1-u_2\,, \qquad \Sigma\, :=\, \o_1-\o_2\,, \qquad B\,:=\, \bt_1-\bt_2
\end{equation*}
and assume that all these quantities belong to $C\big([0,T];L^2(\Omega)\big)$. Define the energy norm
\begin{equation*}
	\mathbb{E}(t)\,:=\, \norm{U(t)}_{L^2}^2\, +\, \norm{\Sigma(t)}_{L^2}^2\, + \,\norm{B(t)}_{L^2}^2\,.
\end{equation*}
Then, there exists a constant $C = C(\nu, \alpha_1, \dots, \alpha_4)>0$, depending only on the quantities inside the brackets,
and a function $\Theta \in L^1\big([0,T]\big)$ such that the following stability estimate holds true:
\begin{equation*}
\forall\,t\in[0,T]\,,\qquad\quad	\mathbb{E}(t)\, \leq\, \mathbb{E}(0)\,\exp \left( C \int_0^t \Theta(\tau)\,\dd \tau \right)\,.
\end{equation*}
\end{thm}

\begin{proof}
The proof follows the same strategy as Theorem 4.2 in \cite{F-GB}. First of all, we find that $U$ solves the following equation,
	\begin{equation}\label{eq:U1}
		\d_t U + u_1\cdot \nabla U + \nabla(\pi_1-\pi_2) - \nu \div \left( \frac{k_1}{\o_1} \D U\right) = - U\cdot \nabla u_2 + \nu \div \left( P \; \D u_2  \right)\,,
	\end{equation}
where $\pi_j$ stands for the pressure associated to the velocity $u_j$, as well as its ``symmetric'' version
\begin{equation}\label{eq:U2}
	\d_t U + u_2\cdot \nabla U +  \nabla(\pi_2-\pi_1)- \nu \div \left( \frac{k_2}{\o_2} \D U\right) = - U\cdot \nabla u_1 + \nu \div \left( P \; \D u_1  \right)\,,
\end{equation}
with the function $P$ defined as
\begin{equation*}
	P \,:=\, \frac{\bt_1^2}{\o_1} - \frac{\bt_2^2}{\o_2}\, =\,-\frac{\bt_1^2}{\o_1 \o_2} \Sigma + \frac{1}{\o_2} (\bt_1+\bt_2) B\,.
\end{equation*}

From the properties of the functions belonging to the space $\wtilde{\mbb X}_T$, it is easy to see that each pressure gradient $\nabla\pi_j$, for $j=1,2$,
belongs to the space $L^\infty_T(H^{-1})$. This can be seen from the analogue of equation \eqref{elliptic}, by observing that the right-hand side is
$L^\infty_T(H^{-1})$ and the dyadic characterisation \eqref{eq:LP-Sob} of Sobolev norms. Hence, also their difference $\nabla\big(\pi_1-\pi_2\big)$
belongs to $L^\infty_T(H^{-1})$. On the other hand, it follows from the assumptions on $u_1$ and $u_2$ that $U\in L^\infty_T(H^1)$,
so the pairing $\lan \nabla\big(\pi_1-\pi_2\big)\,,\,U\ran$ is well-defined, and vanishes owing to the divergence-free constraint $\div U=0$.
Therefore, we can safely perform energy estimates for
\eqref{eq:U1} and \eqref{eq:U2}, integrate by parts when it is useful and sum the two symmetric estimates: we get 
\begin{align*}
&\frac{\dd}{\dt} \norm{U}_{L^2}^2 + \nu \int_{\Omega} \left( \frac{\bt_1^2}{\o_1} + \frac{\bt_2^2}{\o_2} \right) \seminorm{\D U}^2 \dx \\
&\qquad\leq\,
\left(\norm{\nabla u_1}_{L^\infty} + \norm{\nabla u_2}_{L^\infty} \right) \norm{U}_{L^2}^2 \\
& \qquad\qquad +
\nu \seminorm{\int_\Omega \frac{\bt_1^2}{\o_1 \o_2} \Sigma \left( \D u_1 + \D u_2 \right): \D U \dx  } +  \nu \seminorm{\int_\Omega \frac{1}{\o_2} (\bt_1+\bt_2) B \left( \D u_1 + \D u_2 \right): \D U \dx  } \,.
\end{align*}
By rather straightforward computations, we can estimate 
\begin{align*}
&\nu\seminorm{\int_\Omega  \frac{\bt_1^2}{\o_1 \o_2} \Sigma \left( \D u_1 + \D u_2 \right): \D U \dx  } \\
&\qquad\qquad \leq\,\nu\, \left( \norm{\nabla u_1}_{L^\infty} + \norm{\nabla u_2}_{L^\infty}  \right) \norm{\frac{\bt_1}{\o_2 \sqrt{\o_1}}}_{L^\infty} \norm{\Sigma}_{L^2} \norm{\frac{\bt_1}{\sqrt{\o_1}} \D U  }_{L^2} \\
& \qquad\qquad \leq\,\nu\, \delta \int_{\Omega} \frac{\bt_1^2}{\o_1} |\D U|^2 \dx +
C(\delta,\nu)\left( \norm{\nabla u_1}_{L^\infty}^2 + \norm{\nabla u_2}_{L^\infty}^2  \right) \norm{\frac{\bt_1}{\o_2 \sqrt{\o_1}}}_{L^\infty}^2 \norm{\Sigma}_{L^2}^2\,,
\end{align*}
where $\de>0$ will be fixed later and the constant $C(\de,\nu)>0$ depends only on the values of $\de$ and $\nu$.
Moreover, for $j = 1,2$, one has
\begin{align*}
&\nu\seminorm{\int_\Omega  \frac{\bt_j }{\o_2} B \left( \D u_1 + \D u_2 \right): \D U \dx  } \\
& \qquad\qquad   \leq\,\nu \left(\norm{\nabla u_1}_{L^\infty} + \norm{\nabla u_2}_{L^\infty} \right) \norm{\frac{\sqrt{\o_j}}{\o_2}}_{L^\infty} \norm{B}_{L^2} \norm{\frac{\sqrt{\bt_j}}{\sqrt{\o_j}} \D U }_{L^2} \\
	& \qquad\qquad  \leq\,\nu\,\delta  \int_{\Omega} \frac{\bt_j^2}{\o_j} |\D U|^2 \dx + C(\delta,\nu)  \left( \norm{\nabla u_1}_{L^\infty}^2 + \norm{\nabla u_2}_{L^\infty}^2  \right) \norm{\frac{\sqrt{\o_j}}{\o_2}}_{L^\infty}^2 \norm{B}_{L^2}^2\,.
\end{align*}
Consequently, we get the estimate 
\begin{equation}\label{est:U}
\frac{\dd}{\dt} \norm{U}_{L^2}^2\, +\,\nu\, (1-3\delta) \int_{\Omega} \left( \frac{\bt_1^2}{\o_1} + \frac{\bt_2^2}{\o_2} \right) \seminorm{\D U}^2 \dx \lesssim
\Theta_1(t)\, \mathbb{E}(t)\,,
\end{equation}
where the (implicit) multiplicative constant only depends on the parameter $\nu$ and the function $\Theta_1(t)$ is defined as 
\begin{equation*}
\Theta_1(t)\, :=\,\norm{\big(\nabla u_1, \nabla u_2\big)}_{L^\infty}\,  +\,
\norm{\big(\nabla u_1, \nabla u_2\big)}_{L^\infty}^2 \left(   \norm{\frac{\bt_1}{\o_2 \sqrt{\o_1}}}_{L^\infty}^2 +  \norm{\frac{\sqrt{\o_1}}{\o_2}}_{L^\infty}^2 +  \norm{\frac{1}{\o_2}}_{L^\infty} \right).
\end{equation*}

Next, we move on to the equation for $\Sigma$, in order to get a similar $L^2$ estimate for this quantity.
It is easy to check that $\Sigma$ solves the symmetric equations
\begin{align*}\label{eq:sig1}
	\d_t \Sigma + u_1\cdot \nabla \Sg  - \alpha_1 \, \div \left( \frac{k_1}{\o_1} \nabla \Sg \right) + \alpha_2 (\o_1+\o_2) \Sg
&= - U\cdot \nabla \o_2 + \alpha_1 \, \div \left( P \; \nabla \o_2  \right)\,, \\
\d_t \Sigma + u_2\cdot \nabla \Sg  - \alpha_1 \, \div \left( \frac{k_2}{\o_2} \nabla \Sg \right) + \alpha_2 (\o_1+\o_2) \Sg
&= - U\cdot \nabla \o_1 + \alpha_1 \, \div \left( P \; \nabla \o_1  \right).
\end{align*}
Then, we can perform similar computations as above to find 
\begin{equation}\label{est:sig}
\frac{\dd}{\dt} \norm{\Sg}_{L^2}^2 + \alpha_1 (1-3\delta) \int_{\Omega} \left( \frac{\bt_1^2}{\o_1} + \frac{\bt_2^2}{\o_2} \right) \seminorm{\nabla \Sg}^2 \dx + \alpha_2 \int_{\Omega} (\o_1 + \o_2) |\Sg|^2 \dx  \lesssim \Theta_2(t)\, \mathbb{E}(t)\,,
\end{equation}
where the multiplicative constant depends only on $\alpha_1$ and, this time, we have defined
\begin{equation*}
	\Theta_2(t)\, :=\, \norm{\big(\nabla \o_1, \nabla \o_2\big)}_{L^\infty}  +\norm{\big(\nabla \o_1, \nabla \o_2\big)}_{L^\infty}^2 \left(   \norm{\frac{\bt_1}{\o_2 \sqrt{\o_1}}}_{L^\infty}^2 +  \norm{\frac{\sqrt{\o_1}}{\o_2}}_{L^\infty}^2 +  \norm{\frac{1}{\o_2}}_{L^\infty} \right).
\end{equation*}

Finally, we need to find a $L^2$ estimate for $B$. For notational convenience, let us introduce the quantity
\begin{equation*}
	\wtilde{P}\, :=\, \frac{\bt_1}{\o_1} - \frac{\bt_2}{\o_2} = -\frac{\bt_1}{\o_1 \o_2} \Sg + \frac{1}{\o_2} B\,,
\end{equation*}
which is needed to express the extra terms appearing in the equation for $\bt_1$ and $\bt_2$.
We find the following two symmetric equations for $B$:
\begin{align*}
\d_t B + u_1\cdot \nabla B - \alpha_3 \, \div \left( \frac{\bt_1^2}{\o_1} \nabla B\right) + \frac{1}{2} \o_1 B
&= - U\cdot \nabla \bt_2 + \alpha_3 \, \div \left( {P} \, \nabla \bt_2  \right)- \frac{1}{2} \Sg \bt_2 \\
&\qquad + \frac{\alpha_4 }{2} \frac{\bt_1}{\o_1} \D U : (\D u_1 + \D u_2) + \frac{\alpha_4}{2} \wtilde{P} | \D u_2 |^2 \\
&\qquad + \alpha_3 \frac{\bt_1}{\o_1} \nabla B : (\nabla \bt_1 + \nabla \bt_2) + \alpha_3 \wtilde{P} |\nabla \bt_2|^2\,, \\
	\d_t B + u_2\cdot \nabla B - \alpha_3 \, \div \left( \frac{\bt_2^2}{\o_2} \nabla B\right) + \frac{1}{2} \o_2 B  &=
- U\cdot \nabla \bt_1 + \alpha_3 \, \div \left( {P} \, \nabla \bt_1  \right)- \frac{1}{2} \Sg \bt_1 \\
&\qquad + \frac{\alpha_4 }{2} \frac{\bt_2}{\o_2} \D U : (\D u_1 + \D u_2) + \frac{\alpha_4}{2} \wtilde{P} | \D u_1 |^2 \\
&\qquad + \alpha_3 \frac{\bt_2}{\o_2} \nabla B : (\nabla \bt_1 + \nabla \bt_2) + \alpha_3 \wtilde{P} |\nabla \bt_1|^2\,.
\end{align*}
We notice that the first two terms in the right hand side are analogous to the terms appearing in the previous equations, hence they can be estimated in the
same way. The third terms, instead, appears in the energy estimate as a contribution of the type
\begin{equation*}
\seminorm{\int_{\Omega} (\bt_1+\bt_2) \Sg B \dx} \lesssim \left(\norm{\bt_1}_{L^\infty} + \norm{\bt_2}_{L^\infty}\right)\, \mathbb{E}(t)\,.
\end{equation*}
The fourth terms can be estimated as 
\begin{align*}
\seminorm{\int_{\Omega} \frac{\alpha_4 }{2} \frac{\bt_j}{\o_j} \D U : (\D u_1 + \D u_2) B \dx }\, &\leq\,
\nu\,\delta \int_{\Omega} \frac{\bt_j^2}{\o_j} |\D U|^2 \dx \\
&\qquad\qquad + C(\delta,\alpha_4,\nu)\,\norm{ \frac{1}{\o_j}}_{L^\infty} \norm{\big(\nabla u_1,\nabla u_2\big)}_{L^\infty}^2 \norm{B}_{L^2}^2\,.
\end{align*}
Similarly, the sixth terms can be bounded as follows:
\begin{align*}
\seminorm{\int_{\Omega} \alpha_3  \frac{\bt_j}{\o_j} \nabla B : (\nabla \bt_1 + \nabla \bt_2) B \dx }\,&\leq\,
\alpha_3\,\delta \int_{\Omega} \frac{\bt_j^2}{\o_j} |\nabla B|^2 \dx \\
&\qquad\qquad + C(\delta,\alpha_3)\,\norm{ \frac{1}{\o_j}}_{L^\infty} \norm{\big(\nabla \bt_1,\nabla \bt_2\big)}_{L^\infty}^2 \norm{B}_{L^2}^2\,.
\end{align*}
Finally, we can easily control the terms where $\wtilde{P}$ appears. Indeed, we have 
\begin{equation*}
\seminorm{\int_{\Omega} \left( \frac{\alpha_4}{2} \wtilde{P} |\D u_j|^2 + \alpha_3 \wtilde{P} |\nabla \bt_j |^2 \right) B \dx } \lesssim 
\norm{\big(\nabla u_j, \nabla \bt_j\big)}_{L^\infty}^2 \left( \norm{\frac{\bt_1}{\o_1 \o_2}}_{L^\infty} +\norm{ \frac{1}{\o_2} }_{L^\infty}  \right) \mathbb{E}(t)\,.
\end{equation*}

Gathering all the previous estimates, we deduce the following bound for $B$:
\begin{align}
\label{est:B}
&\frac{\dd}{\dt} \norm{B}_{L^2}^2 + \alpha_3(1-3\delta) \int_{\Omega} \left( \frac{\bt_1^2}{\o_1} + \frac{\bt_2^2}{\o_2} \right) \seminorm{\nabla B}^2 \dx + \frac{1}{2} \int_{\Omega} (\o_1 + \o_2) |B|^2 \dx \\
\nonumber
&\qquad\qquad\qquad\qquad\qquad\qquad\qquad\qquad
\leq C\, \Theta_3(t)\, \mathbb{E}(t) \,+\,\nu\,\delta\, \int_{\Omega} \left( \frac{\bt_1^2}{\o_1} +  \frac{\bt_2^2}{\o_2} \right) \seminorm{\D U }^2 \dx\,,
\end{align}
where the constant $C>0$ only depends on the various parameters $\big(\nu,\alpha_1,\ldots \alpha_4\big)$ and where we have set
\begin{align*}
\Theta_3(t)\, &:=\, \norm{\big(\nabla u_1,\nabla u_2, \nabla \bt_1, \nabla \bt_2\big)}_{L^\infty} + \norm{\big(\bt_1,\bt_2\big)}_{L^\infty} \\
& \qquad +  \norm{\big(\nabla u_1,\nabla u_2, \nabla \bt_1, \nabla \bt_2\big)}_{L^\infty}^2 \\
&\qquad\qquad\qquad \times\, \left( \norm{\frac{\bt_1}{\sqrt{\o_1} \o_2}  }_{L^\infty}^2+ \norm{\frac{\sqrt{\o_1}}{\o_2}  }_{L^\infty}^2+    \norm{\left( \frac{1}{\o_1},  \frac{1}{\o_2}\right)}_{L^\infty} + \norm{\frac{\bt_1}{\o_1 \o_2}  }_{L^\infty} \right).
\end{align*}

To conclude the argument, we sum up estimates \eqref{est:U}, \eqref{est:sig} and \eqref{est:B}. Fixing the value of $\delta >0$ small enough to absorb the extra
terms within the left hand side of the obtained inequality, in the end we find the estimate
\begin{equation*}
\frac{\dd}{\dt} \mathbb{E}(t)\, \lesssim\, \Big(\Theta_1(t) + \Theta_2(t) + \Theta_3(t)\Big)\, \mathbb{E}(t)\,,
\end{equation*}
where the multiplicative constant $C = C(\nu,\alpha_1,\dots \alpha_4)>0$ only depends on the quantities appearing inside the brackets.
At this point, an application of the Gr\"onwall lemma concludes the proof of the theorem.
\end{proof}


\end{document}